\documentclass{amsart}%
\usepackage{amsfonts}
\usepackage{amsmath}
\usepackage{amssymb}
\usepackage[unicode=true]{hyperref}
\usepackage{graphicx}%
\providecommand{\U}[1]{\protect\rule{.1in}{.1in}}
\newtheorem{theorem}{Theorem}
\theoremstyle{plain}

\newtheorem{example}{Example}

\newtheorem{lemma}{Lemma}

\newtheorem{remark}{Remark}

\numberwithin{equation}{section}
\begin{document}
\title[Oscillations: self improvement]{BMO: Oscillations, self improvement, Gagliardo coordinate spaces and reverse
Hardy inequalities}
\author{Mario Milman}
\address{Instituto Argentino de Matematica}
\email{mario.milman@gmail.com}
\urladdr{https://sites.google.com/site/mariomilman}
\thanks{The author was partially supported by a grant from the Simons Foundation
(\#207929 to Mario Milman)}

\begin{abstract}
A new approach to classical self improving results for $BMO$ functions is
presented. \textquotedblleft Coordinate Gagliardo spaces" are introduced and a
generalized version of the John-Nirenberg Lemma is proved. Applications are provided.

\end{abstract}
\dedicatory{Para Corita\footnote{See last Section: *a brief personal note on Cora
Sadosky*.}}\maketitle
\tableofcontents

\section{Introduction and Background}

Interpolation theory provides a framework, as well as an arsenal of tools,
that can help in our understanding of the properties of function spaces and
the operators acting on them. Conversely, the interaction of the abstract
theory of interpolation with concrete function spaces can lead to new general
methods and results. In this note we consider some aspects of the interaction
between interpolation theory and $BMO,$ focussing on the self improving
properties of $BMO$ functions.

To fix the notation, in this section we shall consider functions defined on a
fixed cube, $Q_{0}\subset R^{n}.$ A prototypical example of the
self-improvement exhibited by $BMO$ functions is the statement that a function
in $BMO$ automatically belongs to all $L^{p}$ spaces, $p<\infty,$%
\begin{equation}
BMO\subset%
%TCIMACRO{\dbigcap \limits_{p\geq1}}%
%BeginExpansion
{\displaystyle\bigcap\limits_{p\geq1}}
%EndExpansion
L^{p}. \label{rates0}%
\end{equation}
In fact, $BMO$ is contained in the Orlicz space $e^{L}.$ This is one of the
themes underlying the John-Nirenberg Lemma \cite{johnN}. One way to obtain
this refinement is to make explicit the rates of decay of the family of
embeddings implied by (\ref{rates0}).

We consider in detail an inequality apparently first shown in \cite{chenzu},%
\begin{equation}
\left\Vert f\right\Vert _{L^{q}}\leq C_{n}q\left\Vert f\right\Vert _{L^{p}%
}^{p/q}\left\Vert f\right\Vert _{BMO}^{1-p/q},1\leq p<q<\infty. \label{intro1}%
\end{equation}
With (\ref{intro1}) at hand we can, for example, extrapolate by the $\Delta
-$method of \cite{jm}, and the exponential integrability of $BMO$ functions
follows (cf. (\ref{expo}) below)%
\begin{equation}
\left\Vert f\right\Vert _{e^{L}}\sim\sup_{q>1}\frac{\left\Vert f\right\Vert
_{L^{q}}}{q}\leq c_{n}\left\Vert f\right\Vert _{BMO}. \label{linfty0}%
\end{equation}
More generally, for compatible Banach spaces, interpolation inequalities of
the form%
\begin{equation}
\left\Vert f\right\Vert _{X}\leq c(\theta)\left\Vert f\right\Vert _{X_{1}%
}^{1-\theta}\left\Vert f\right\Vert _{X_{2}}^{\theta},\text{ }\theta\in(0,1),
\label{intro2}%
\end{equation}
where $c(\theta)$ are constants that depend only on $\theta,$ play an
important role in analysis. What is needed to extract information at the end
points (e.g. by \textquotedblleft extrapolation" \cite{jm}) is to have good
estimates of the rate of decay $c(\theta),$ as $\theta$ tends to $0$ or to
$1.$ We give a brief summary of inequalities of the form (\ref{intro2}) for
the classic methods of interpolation in Section \ref{secc:inter} below. For
example, a typical interpolation inequality of the form (\ref{intro2}) for the
Lions-Peetre real interpolation spaces can be formulated as follows. Given a
compatible pair\footnote{We refer to Section \ref{secc:inter} for more
details.} of Banach spaces $\vec{X}=(X_{1},X_{2})$, the \textquotedblleft%
$K-$functional\textquotedblright\ is defined for $f\in\Sigma(\vec{X}%
)=X_{1}+X_{2},t>0,$ by
\begin{equation}
K(t,f;\vec{X}):=K(t,f;X_{1},X_{2})=\inf_{f=f_{1}+f_{2},f_{i}\in X_{i}%
}\{\left\Vert f_{1}\right\Vert _{X_{1}}+t\left\Vert f_{2}\right\Vert _{X_{2}%
}\}. \label{k1}%
\end{equation}
The real interpolation spaces $\vec{X}_{\theta,q}$ can be defined through the
use of the $K-$functional. Let $\theta\in(0,1),0<q\leq\infty,$ then we let%
\begin{equation}
\vec{X}_{\theta,q}=\{f\in\Sigma(\vec{X}):\left\Vert f\right\Vert _{\vec
{X}_{\theta,q}}<\infty\}, \label{arriba}%
\end{equation}
where\footnote{with the usual modification when $q=\infty.$}%
\[
\left\Vert f\right\Vert _{\vec{X}_{\theta,q}}=\left\{  \int_{0}^{\infty
}\left[  t^{-\theta}K(t,f;\vec{X})\right]  ^{q}\frac{dt}{t}\right\}  ^{1/q}.
\]
We have (cf. Lemma \ref{marcaada} below) that, for $f\in X_{1}\cap
X_{2},0<\theta<1,1\leq q\leq\infty$,%
\begin{equation}
\lbrack(1-\theta)\theta q]^{1/q}\left\Vert f\right\Vert _{(X_{1}%
,X_{2})_{\theta,q}}\leq\left\Vert f\right\Vert _{X_{1}}^{1-\theta}\left\Vert
f\right\Vert _{X_{2}}^{\theta}. \label{rates}%
\end{equation}
We combine (\ref{rates}), the known real interpolation theory of $BMO$ (cf.
\cite[Theorem 6.1]{besa}, \cite{bs} and the references therein), sharp reverse
Hardy inequalities (cf. \cite{ren} and \cite{mil}), and the re-scaling of
inequalities via the reiteration method (cf. \cite{bl}, \cite{holm0}) to give
a new *interpolation* proof of (\ref{intro1}) in Lemma \ref{marcao} below.

Let us now recall how the study of $BMO$ led to new theoretical developments
in interpolation theory\footnote{Paradoxically, except for Section
\ref{secc:bilinear}, in this paper we do not discuss interpolation theorems
per se. For interpolation theorems involving $BMO$ type of spaces there is a
large literature. For articles that are related to the developments in this
note I refer, for example, to \cite{herz}, \cite{bedesa}, \cite{sag},
\cite{misa}, \cite{jm1}, \cite{kru}.}.

A natural follow up question to (\ref{linfty0}) was to obtain the best
possible integrability condition satisfied by $BMO$ functions. The answer was
found by Bennett-DeVore-Sharpley \cite{bedesa}. They showed the
inequality\footnote{where $f^{\ast}$ denotes the non-increasing rearrangement
of $f$ and $f^{\ast\ast}(t)=\frac{1}{t}\int_{0}^{t}f^{\ast}(s)ds.$}
\begin{equation}
\left\Vert f\right\Vert _{L(\infty,\infty)}:=\sup_{t}\left\{  f^{\ast\ast
}(t)-f^{\ast}(t)\right\}  \leq c_{n}\left\Vert f\right\Vert _{BMO}.
\label{linfty}%
\end{equation}
The refinement here is that the (non-linear) function space $L(\infty
,\infty),$ defined by the condition
\[
\left\Vert f\right\Vert _{L(\infty,\infty)}<\infty,
\]
is strictly contained\footnote{The smallest rearrangement invariant space that
contains $BMO$ is $e^{L}$ as was shown by Pustylnik \cite{pus}.} in $e^{L}.g$

In their celebrated work, Bennett-DeVore-Sharpley \cite{bedesa} proposed the
following connection between real interpolation, weak interpolation, and $BMO$
(cf. \cite[page 384]{bs}). The $K-$functional for the pair $(L^{1},L^{\infty
})$ is given by (cf. \cite{bs} and Section \ref{secc:bgextra} below)%
\[
K(t,f;L^{1},L^{\infty})=\int_{0}^{t}f^{\ast}(s)ds.
\]
Therefore, $\frac{dK(t,f;L^{1},L^{\infty})}{dt}=K^{\prime}(t,f;L^{1}%
,L^{\infty})=f^{\ast}(t);$ consequently, we can compute the \textquotedblleft
norm" of weak $L^{1}:=L(1,\infty)$, as follows%
\begin{align*}
\left\Vert f\right\Vert _{L(1,\infty)}  &  =\sup_{t>0}tf^{\ast}(t)\\
&  =\sup_{t>0}tK^{\prime}(t,f;L^{1},L^{\infty}).
\end{align*}
Then, in analogy with the definition of weak $L^{1},$ Bennett-DeVore-Sharpley
proceeded to define $L(\infty,\infty)$ using the functional%
\begin{equation}
\left\Vert f\right\Vert _{L(\infty,\infty)}:=\sup_{t>0}tK^{\prime
}(t,f;L^{\infty},L^{1}). \label{berta4}%
\end{equation}
Note that in (\ref{berta4}) the order of the spaces is reversed in the
computation of the $K-$functional. These two different $K-$functionals are
connected by the equation%
\begin{equation}
K(t,f;L^{\infty},L^{1})=tK(\frac{1}{t},f;L^{1},L^{\infty}). \label{ladada}%
\end{equation}
Inserting (\ref{ladada}) in (\ref{berta4}) we readily see that%
\begin{align*}
\left\Vert f\right\Vert _{L(\infty,\infty)}  &  =\sup_{t>0}\{tK(\frac{1}%
{t},f;L^{1},L^{\infty})-K^{\prime}(\frac{1}{t},f;L^{1},L^{\infty})\}\\
&  =\sup_{t>0}\{\frac{K(t,f;L^{1},L^{\infty})}{t}-K^{\prime}(t,f;L^{1}%
,L^{\infty})\}\\
&  =\sup_{t}\{f^{\ast\ast}(t)-f^{\ast}(t)\}.
\end{align*}

The oscillation operator, $f\rightarrow f^{\ast\ast}(t)-f^{\ast}(t),$ turns
out to play an important role in other fundamental inequalities in analysis. A
recent remarkable application of the oscillation operator provides the sharp
form of the Hardy-Littlewood-Sobolev-O'Neil inequality up the borderline end
point $p=n.$ Indeed, if we let%
\begin{equation}
\left\Vert f\right\Vert _{L(p,q)}=\left\{
\begin{array}
[c]{cc}%
\left\{  \int_{0}^{\infty}\left(  f^{\ast}(t)t^{1/p}\right)  ^{q}\frac{dt}%
{t}\right\}  ^{1/q} & 1\leq p<\infty,1\leq q\leq\infty\\
\left\Vert f\right\Vert _{L(\infty,q)} & 1\leq q\leq\infty,
\end{array}
\right.  \label{berta}%
\end{equation}
where\footnote{Apparently the $L(\infty,q)$ spaces for $q<\infty$ were first
introduced and their usefulness shown in \cite{bmr}. Note that with the usual
definition $L(\infty,\infty)$ would be $L^{\infty},$ and $L(\infty,q)=\{0\},$
for $q<\infty.$ The key point here is that the use of the oscillation operator
introduces cancellations that make the spaces defined in this fashion
non-trivial (cf. Section \ref{secc:coordenadas}, Example \ref{ejemplomarkao}%
).}%
\begin{equation}
\left\Vert f\right\Vert _{L(\infty,q)}:=\left\{  \int_{0}^{\infty}(f^{\ast
\ast}(t)-f^{\ast}(t))^{q}\frac{dt}{t}\right\}  ^{1/q}, \label{berta1}%
\end{equation}
then it was shown in \cite{bmr} that
\begin{equation}
\left\Vert f\right\Vert _{L(\bar{p},q)}\leq c_{n}\left\Vert \nabla
f\right\Vert _{L(p,q)},1\leq p\leq n,\frac{1}{\bar{p}}=\frac{1}{p}-\frac{1}%
{n},\text{ }f\in C_{0}^{\infty}(R^{n}). \label{sobolev}%
\end{equation}
The Sobolev inequality (\ref{sobolev}) is best possible, and for $p=q=n$ it
improves on the end point result of
Brezis-Wainger-Hanson-Maz'ya\footnote{which in turn improves upon the
classical exponential integrability result by Trudinger \cite{tru}.} (cf.
\cite{brez}, \cite{ha}, \cite{maz}) much as the Bennett-DeVore-Sharpley
inequality (\ref{linfty}) improves upon (\ref{linfty0}). The improvement over
*best possible results* is feasible because, once again, the spaces that
correspond to $p=n,$ i.e.
\[
L(\infty,q)=\{f:\left\Vert f\right\Vert _{L(\infty,q)}<\infty\},
\]
are not necessarily linear\footnote{Let $X$ be a rearrangement invariant
space, Pustylnik \cite{pus} has given necessary and sufficient conditions for
spaces of functions defined by conditions of the form%
\[
\left\Vert \left(  f^{\ast\ast}-f^{\ast}\right)  t^{-\gamma}\right\Vert
_{X}<\infty
\]
to be linear and normable.}!

Moreover, the Sobolev inequality (\ref{sobolev}) persists up to higher order
derivatives\footnote{The improvement is also valid for Besov space
inequalities as well (cf. \cite{mamipams}).}, as was shown in \cite{milpu},%
\[
\left\Vert f\right\Vert _{L(\bar{p},q)}\leq c_{n}\left\Vert \nabla
^{k}f\right\Vert _{L(p,q)},1\leq p\leq\frac{n}{k},\frac{1}{\bar{p}}=\frac
{1}{p}-\frac{k}{n},\text{ }f\in C_{0}^{\infty}(R^{n}).
\]
In particular, when $p=\frac{n}{k}$ and $q=\infty,$ we have the $BMO$ type
result\footnote{The spaces $L(\infty,q)$ allow to interpolate between
$L^{\infty}=L(\infty,1)$ and $L(\infty,\infty)\subset e^{L}.$}%
\[
\left\Vert f\right\Vert _{L(\infty,\infty)}\leq c\left\Vert \nabla
^{k}f\right\Vert _{L(\frac{n}{k},\infty)},f\in C_{0}^{\infty}(R^{n}).
\]

Using the space $L(\infty,\infty)$ one can improve (\ref{intro1}) as follows
(cf. \cite{kowa})%
\begin{equation}
\left\Vert f\right\Vert _{L^{q}}\leq C_{n}q\left\Vert f\right\Vert _{L^{p}%
}^{p/q}\left\Vert f\right\Vert _{L(\infty,\infty)}^{1-p/q},1\leq p<q<\infty.
\label{lageneral}%
\end{equation}

In my work with Jawerth\footnote{The earlier work of Herz \cite{herz} and
Holmstedt \cite{holm}, that precedes \cite{bedesa}, should be also mentioned
here.} (cf. \cite{jm1}) we give a somewhat different interpretation of the
$L(\infty,q)$ spaces using Gagliardo diagrams (cf. \cite{bl}, \cite{cwbjmm0});
this point of view turns out to be useful to explain other applications of the
oscillation operator $f^{\ast\ast}(t)-f^{\ast}(t)$ (cf. \cite{cwbjmm},
\cite{mamicoc} and Section \ref{secc:uses}). The idea behind the approach in
\cite{jm1} is that of an \textquotedblleft optimal decomposition", which also
makes it possible to incorporate the $L(\infty,q)$ spaces into the abstract
theory of real interpolation, as we shall show below.

Let $t>0,$ and let $f\in\Sigma(\vec{X})=X_{1}+X_{2}.$ Out of all the competing
decompositions for the computation of $K(t,f;\vec{X}),$ an
optimal\footnote{Optimal decompositions exist for $a.e.$ $t>0.$}
decomposition
\[
f=D_{1}(t)f+D_{2}(t)f,\text{ with }D_{i}(t)f\in X_{i},i=1,2,
\]
satisfies%
\begin{equation}
K(t,f;\vec{X})=\left\Vert D_{1}(t)f\right\Vert _{X_{1}}+t\left\Vert
D_{2}(t)f\right\Vert _{X_{2}}, \label{k3}%
\end{equation}
(resp. a nearly optimal decomposition obtains if in (\ref{k3}) we replace $=$
by $\approx).$ For an optimal decomposition of $f$ we
have\footnote{Interpreting $(\left\Vert D_{1}(t)f\right\Vert _{X_{1}%
},\left\Vert D_{2}(t)f\right\Vert _{X_{2}})$ as coordinates on the boundary of
a Gagliardo diagram (cf. \cite{bl}) it follows readily that, for all
$\varepsilon>0,t>0,$ we can find nearly optimal decompositions
$x=x_{\varepsilon}(t)+y_{\varepsilon}(t),$ such that%
\begin{equation}
(1-\varepsilon)[K(t,f;\vec{X})-t\frac{d}{dt}K(t,f;\vec{X})]\leq\left\Vert
x_{\varepsilon}(t)\right\Vert _{X_{1}}\leq(1+\varepsilon)[K(t,f;\vec
{X})-t\frac{d}{dt}K(t,f;\vec{X})] \label{a1'}%
\end{equation}%
\begin{equation}
(1-\varepsilon)\frac{d}{dt}K(t,f;\vec{X})\leq\left\Vert y_{\varepsilon
}(t)\right\Vert _{X_{2}}\leq(1+\varepsilon)\frac{d}{dt}K(t,f;\vec{X}).
\label{a2'}%
\end{equation}
} (cf. \cite{holm}, \cite{jm1})
\begin{equation}
\left\Vert D_{1}(t)f\right\Vert _{X_{1}}=K(t,f;\vec{X})-t\frac{d}%
{dt}K(t,f;\vec{X}); \label{a1}%
\end{equation}%
\begin{equation}
\left\Vert D_{2}(t)f\right\Vert _{X_{2}}=\frac{d}{dt}K(t,f;\vec{X}).
\label{a2}%
\end{equation}

In particular, for the pair $(L^{1},L^{\infty}),$ we have (cf. Section
\ref{secc:bgextra}),%
\begin{align*}
\left\Vert D_{1}(t)f\right\Vert _{L^{1}}  &  =K(t,f;L^{1},L^{\infty}%
)-t\frac{d}{dt}K(t,f;L^{1},L^{\infty})\\
&  =tf^{\ast\ast}(t)-tf^{\ast}(t),
\end{align*}
and%
\[
\left\Vert D_{2}(t)f\right\Vert _{L^{\infty}}=\frac{d}{dt}K(t,f;L^{1}%
,L^{\infty})=f^{\ast}(t).
\]
Thus, reinterpreting optimal decompositions using Gagliardo diagrams (cf.
\cite{holm}, \cite{jm1}), one is led to consider spaces, which could be
referred to as \textquotedblleft Gagliardo coordinate spaces". These spaces
coincide with the Lions-Peetre real interpolation spaces, for the usual range
of the parameters (cf. \cite{holm}), but they also make sense at the end
points (cf. \cite{jm1}), and in this fashion they can be used to complete the
Lions-Peetre scale much as the generalized $L(p,q)$ spaces, defined by
(\ref{berta}), complete the classical scale of Lorentz spaces. The
\textquotedblleft Gagliardo coordinate spaces" $\vec{X}_{\theta,q}%
^{(i)},i=1,2,$ are formally obtained replacing $K(t,f;\vec{X}),$ in the
definition of the $\vec{X}_{\theta,q}$ norm of $f$ (cf. (\ref{arriba})
above)$,$ by $\left\Vert D_{1}(t)f\right\Vert _{X_{1}}$ (resp. $\left\Vert
D_{2}(t)f\right\Vert _{X_{2}})$ (cf. \cite{holm})$.$ In particular, we note
that the $\vec{X}_{\theta,q}^{(2)}$ spaces correspond to the $k$-spaces
studied by Bennett \cite{beka}.

This point of view led us to formulate and prove the following general version
of (\ref{rates}) (cf. Theorem \ref{teomarkao} below)%
\begin{equation}
\left\Vert f\right\Vert _{\vec{X}_{\theta,q(\theta)}^{(2)}}\leq cq\left\Vert
f\right\Vert _{X_{1}}^{1-\theta}\left\Vert f\right\Vert _{\vec{X}_{1,\infty
}^{(1)}}^{\theta},\theta\in(0,1),1-\theta=\frac{1}{q(\theta)}, \label{berta3}%
\end{equation}
which can be easily seen to imply an abstract extrapolation theorem connected
with the John-Nirenberg inequality.

Underlying these developments is the following computation of the
$K-$functional for the pair $(L^{1},BMO)$ given in \cite{besa} (cf. Section
\ref{secc:self} below)
\begin{equation}
K(t,f,L^{1}(R^{n}),BMO(R^{n}))\approx tf^{\#\ast}(t), \label{obama}%
\end{equation}
where $f^{\#}$ is the sharp maximal function of Fefferman-Stein \cite{fest}
(cf. (\ref{veaabajo}) below). In this calculation $BMO(R^{n})$ is provided
with the seminorm\footnote{$BMO(R^{n})$ can be normed by $\left\vert
f\right\vert _{BMO}$ if we identify functions that differ by a constant.} (cf.
Section \ref{secc:abajo})%
\[
\left\vert f\right\vert _{BMO}=\left\Vert f_{R^{n}}^{\#}\right\Vert
_{L^{\infty}}.
\]
In \cite{chenzu}, the authors show that (\ref{intro1}) can be used to give a
strikingly \textquotedblleft easy" proof of an inequality first proved in
\cite{kowa1} using paraproducts,%
\begin{equation}
\left\Vert fg\right\Vert _{L^{p}}\leq c(\left\Vert f\right\Vert _{L^{p}%
}\left\Vert g\right\Vert _{BMO}+\left\Vert g\right\Vert _{L^{p}}\left\Vert
f\right\Vert _{BMO}),1<p<\infty. \label{intro3}%
\end{equation}
The argument to prove (\ref{intro3}) given in \cite{chenzu} has a general
character and with suitable modifications (in particular, using the
*reiteration theorem* of interpolation theory) the idea\footnote{We cannot
resist but to offer here our slight twist to the argument%
\begin{align*}
\left\Vert fg\right\Vert _{L^{p}}  &  \leq\left\Vert f\right\Vert _{L^{2p}%
}\left\Vert g\right\Vert _{L^{2p}}\\
&  =\left\Vert f\right\Vert _{[L^{p},BMO]_{1/2,2p}}\left\Vert g\right\Vert
_{[L^{p},BMO]_{1/2,2p}}\\
&  \preceq\left\Vert f\right\Vert _{L^{p}}^{1/2}\left\Vert f\right\Vert
_{BMO}^{1/2}\left\Vert g\right\Vert _{L^{p}}^{1/2}\left\Vert g\right\Vert
_{BMO}^{1/2}\\
&  \leq\left\Vert f\right\Vert _{L^{p}}\left\Vert g\right\Vert _{BMO}%
+\left\Vert g\right\Vert _{L^{p}}\left\Vert f\right\Vert _{BMO}.
\end{align*}
} can be combined with (\ref{berta3}) to yield a new end point result for
bilinear interpolation for generalized product and convolution operators of
O'Neil type acting on interpolation scales (see Section \ref{secc:bilinear}
below). Further, in Section \ref{secc:uses} we collect applications of the
methods discussed in the paper, and also offer some
suggestions\footnote{However, keep in mind the epigraph of \cite{restaurant},
originally due Douglas Adams, The Restaurant at the End of the Universe, Tor
Books, 1988 :\textquotedblleft For seven and a half million years, Deep
Thought computed and calculated, and in the end announced that the answer was
in fact Forty-two---and so another, even bigger, computer had to be built to
find out what the actual question was."} for further research. In particular,
in Section \ref{secc:singular} we give a new approach to well known results by
Bagby-Kurtz \cite{bagby}, Kurtz \cite{kurtz}, on the linear (in $p$) rate of
growth of $L^{p}$ estimates for certain singular integrals; in Section
\ref{secc:lambda} we discuss the connection between the classical good lambda
inequalities (cf. \cite{bg}, \cite{cof}) and the strong good lambda
inequalities of Bagby-Kurtz (cf. \cite{kurtz}), with inequalities for the
oscillation operator $f^{\ast\ast}-f^{\ast}$; in Section \ref{secc:bgextra} we
show how oscillation inequalities for Sobolev functions are connected with the
Gagliardo coordinate spaces and the property of commutation of the gradient
with optimal $(L^{1},L^{\infty})$ decompositions (cf. \cite{cwbjmm}), we also
discuss briefly the characterization of the isoperimetric inequality in terms
of rearrangement inequalities for Sobolev functions, in a very general context..

The intended audience for this note are, on the one hand, classical analysts
that may be curious on what abstract interpolation constructions could bring
to the table, and on the other hand, functional analysts, specializing in
interpolation theory, that may want to see applications of the abstract
theories. To balance these objectives I have tried to give a presentation full
of details in what respects to interpolation theory, and provide full
references to the background material needed for the applications to classical
analysis. In this respect, I have compiled a large set of references but the
reader should be warned that this paper is not intended to be a survey, and
that the list intends only to document the material that is mentioned in the
text and simply reflects my own research interests, point of view, and
limitations. In fact, many important topics dear to me had to be left out,
including *Garsia inequalities* (cf. \cite{garsiagrenoble}).

I close the note with some personal reminiscences of my friendship with Cora Sadosky.

\section{The John-Nirenberg Lemma and rearrangements\label{secc:abajo}}

In this section we recall a few basic definitions and results associated with
the self improving properties of $BMO$ functions\footnote{For more background
information we refer to \cite{bs} and \cite{cosa}.}. In particular, we discuss
the John-Nirenberg inequality (cf. \cite{johnN}). In what follows we always
let $Q$ denote a cube.

Let $Q_{0}$ be a fixed cube in $R^{n}$. For $x\in Q_{0},$ let%
\begin{equation}
f_{Q_{0}}^{\#}(x)=\sup_{x\backepsilon Q,Q\subset Q_{0}}\frac{1}{\left\vert
Q\right\vert }\int_{Q}\left\vert f-f_{Q}\right\vert dx,\text{ where }%
f_{Q}=\frac{1}{\left\vert Q\right\vert }\int_{Q}fdx. \label{veaabajo}%
\end{equation}
The space of functions of bounded mean oscillation, $BMO(Q_{0}),$ consists of
all the functions $f\in L^{1}(Q_{0})$ such that $f_{Q_{0}}^{\#}\in L^{\infty
}(Q_{0}).$ Generally, we use the seminorm%
\begin{equation}
\left\vert f\right\vert _{BMO(Q_{0})}=\left\Vert f_{Q_{0}}^{\#}\right\Vert
_{L^{\infty}}. \label{llamada}%
\end{equation}
The space $BMO(Q_{0})$ becomes a Banach space if we identify functions that
differ by a constant. Sometimes it is preferable for us to use
\[
\left\Vert f\right\Vert _{BMO(Q_{0})}=\left\vert f\right\vert _{BMO(Q_{0}%
)}+\left\Vert f\right\Vert _{L^{1}(Q_{0})}.
\]
The classical John-Nirenberg Lemma is reformulated in \cite[Corollary 7.7,
page 381]{bs} as follows\footnote{For a recent new approach to the
John-Nirenberg Lemma we refer to \cite{css}.}: Given a fixed cube
$Q_{0}\subset R^{n},$ there exists a constant $c>0,$ such that for all $f\in
BMO(Q_{0}),$ and for all subcubes $Q\subset Q_{0},$
\begin{equation}
\lbrack\left(  f-f_{Q}\right)  \chi_{Q}]^{\ast}(t)\leq c\left\vert
f\right\vert _{BMO(Q_{0})}\log^{+}(\frac{6\left\vert Q\right\vert }{t}),t>0.
\label{jn1}%
\end{equation}
In particular, $BMO$ has the following self improving property (cf.
\cite[Corollary 7.8, page 381]{bs}). Let $1\leq p<\infty,$ and
let\footnote{Note that $f_{Q_{0},1}^{\#}=f_{Q_{0}}^{\#}.$}%
\[
f_{Q_{0},p}^{\#}(x)=\sup_{x\backepsilon Q,Q\subset Q_{0}}\left\{  \frac
{1}{\left\vert Q\right\vert }\int_{Q}\left\vert f-f_{Q}\right\vert
^{p}dx\right\}  ^{1/p},
\]
and
\[
\left\Vert f\right\Vert _{BMO^{p}(Q_{0})}=\left\Vert f_{Q_{0},p}%
^{\#}\right\Vert _{L^{\infty}}+\left\Vert f\right\Vert _{L^{p}(Q_{0})}.
\]
Then, with constants independent of $f,$%
\[
\left\Vert f\right\Vert _{BMO^{p}(Q_{0})}\approx\left\Vert f\right\Vert
_{BMO(Q_{0})}.
\]
It follows that, for all $p<\infty,$
\[
BMO(Q_{0})\subset L^{p}(Q_{0}).
\]
Actually, from \cite{jm} we have%
\[
\left\Vert \left(  f-f_{Q}\right)  \chi_{Q}\right\Vert _{\Delta(\frac
{L^{p}(Q)}{p})}\approx\sup_{t}\frac{[\left(  f-f_{Q}\right)  \chi_{Q}]^{\ast
}(t)}{\log^{+}(\frac{6\left\vert Q\right\vert }{t})}\approx\left\Vert
f-f_{Q}\right\Vert _{e^{L(Q)}},
\]
which combined with (\ref{jn1}) gives%
\[
\left\Vert f-f_{Q}\right\Vert _{e^{L(Q)}}\leq c\left\vert f\chi_{Q}\right\vert
_{BMO(Q)},
\]
and therefore (cf. \cite{bs})
\[
BMO(Q_{0})\subset e^{L(Q_{0})}.
\]
In other words, the functions in $BMO(Q_{0})$ are exponentially integrable.

The previous results admit suitable generalizations to $R^{n}$ and more
general measure spaces.

\section{Interpolation theory: some basic inequalities\label{secc:inter}}

In this section we review basic definitions, and discuss inequalities of the
form (\ref{intro2}) that are associated with the classical methods of interpolation.

The starting objects of interpolation theory are pairs $\vec{X}=(X_{1},X_{2})$
of Banach spaces that are \textquotedblleft compatible\textquotedblright, in
the sense that both spaces are continuously embedded in a common Hausdorff
topological vector space $V$\footnote{We shall then call $\vec{X}=(X_{0}%
,X_{1})$ a \textquotedblleft Banach pair\textquotedblright. In general, the
space $V$ plays an auxiliary role, since once we know that $\vec{X}$ is a
Banach pair we can use $\Sigma(\vec{X})$ as the ambient space. In particular,
the functional $K(t,f;\vec{X})$ is in principle only defined on $\Sigma
(\vec{X}).$ On the other hand, the functional $f\rightarrow\frac{d}%
{dt}K(t,f;\vec{X}),$ can make sense for a larger class of elements than
$\Sigma(\vec{X}).$ This occurs for significant examples: For example, on the
interval $[0,1],$
\[
L(1,\infty)=\{f:\sup_{t}tf^{\ast}(t)<\infty\}\nsubseteq L^{1}+L^{\infty}%
=L^{1}.
\]
}. In real interpolation we consider two basic functionals, the $K-$%
functional, already introduced in (\ref{k1}), associated with the construction
of the sum space $\Sigma(\vec{X})=X_{1}+X_{2},$ and its counterpart, the
$J-$functional, defined on the intersection space $\Delta(\vec{X})=X_{1}\cap
X_{2},$ by
\begin{equation}
J(t,f;\vec{X}):=J(t,f;X_{1},X_{2})=\max\left\{  \left\Vert f\right\Vert
_{X_{1}},t\left\Vert f\right\Vert _{X_{2}}\right\}  ,t>0. \label{j1}%
\end{equation}
The $K-$functional is used to construct the interpolation spaces $(X_{1}%
,X_{2})_{\theta,q}$ (cf. (\ref{arriba}) above). Likewise, associated with the
$J-$functional we have the $(X_{1},X_{2})_{\theta,q;J},$ spaces. Let
$\theta\in(0,1),1\leq q\leq\infty;$ and let $U_{\theta,q}$ be the class of
functions $u:(0,\infty)\rightarrow\Delta(\vec{X}),$ such that\footnote{with
the usual modification when $q=\infty.$} $\left\Vert u\right\Vert
_{U_{\theta,q}}=\left\{  \int_{0}^{\infty}(t^{-\theta}J(t,u(t);X_{1}%
,X_{2}))^{q}\frac{dt}{t}\right\}  ^{1/q}<\infty$. The space $(X_{1}%
,X_{2})_{\theta,q;J}$ consists of elements $f\in$ $X_{1}+X_{2},$ such that
there exists $u\in U_{\theta,q},$ with%
\[
f=\int_{0}^{\infty}u(s)\frac{ds}{s}\text{ (in }X_{1}+X_{2}),
\]
provided with the norm%
\[
\left\Vert f\right\Vert _{(X_{1},X_{2})_{\theta,q;J}}=\inf_{f=\int_{0}%
^{\infty}u(s)\frac{ds}{s}}\{\left\Vert u\right\Vert _{U_{\theta,q}}\}.
\]
A basic result in this context is that the two constructions give the same
spaces (*the equivalence theorem*) (cf. \cite{bl})%
\[
(X_{1},X_{2})_{\theta,q}=(X_{1},X_{2})_{\theta,q;J},
\]
where the constants of norm equivalence depend only on $\theta$ and $q.$

In practice the $J-$method is harder to compute, but nevertheless plays an
important theoretical role. In particular, the following interpolation
property holds for the $J-$method. If $X$ is a Banach space intermediate
between $X_{1}$ and $X_{2},$ in the sense that $\Delta(\vec{X})\subset
X\subset\Sigma(\vec{X}),$ then an inequality of the form%
\begin{equation}
\left\Vert f\right\Vert _{X}\leq\left\Vert f\right\Vert _{X_{1}}^{1-\theta
}\left\Vert f\right\Vert _{X_{2}}^{\theta},\text{ for some fixed }\theta
\in(0,1),\text{ and for all }f\in\Delta(\vec{X}), \label{propiedadj}%
\end{equation}
is equivalent to
\begin{equation}
\left\Vert f\right\Vert _{X}\leq\left\Vert f\right\Vert _{(X_{1}%
,X_{2})_{\theta,1;J}},\text{ for all }f\in(X_{1},X_{2})_{\theta,1;J}.
\label{propiedaj1}%
\end{equation}
One way to see this equivalence is to observe that (cf. \cite{bs})

\begin{lemma}
\label{lemaviva}Let $\theta\in(0,1).$ Then, for all $f\in X_{1}\cap X_{1},$%
\begin{equation}
\left\Vert f\right\Vert _{X_{1}}^{1-\theta}\left\Vert f\right\Vert _{X_{2}%
}^{\theta}=\inf_{t>0}\{t^{-\theta}J(t,f;X_{1},X_{2})\}. \label{lect14}%
\end{equation}

\end{lemma}

The preceding discussion shows that, in particular,%
\[
\left\Vert f\right\Vert _{(X_{1},X_{2})_{\theta,1,J}}\leq\left\Vert
f\right\Vert _{X_{1}}^{1-\theta}\left\Vert f\right\Vert _{X_{2}}^{\theta
},0<\theta<1,\text{ for }f\in X_{1}\cap X_{2}.
\]
More generally (cf. \cite{jm}, \cite{milta}) we have\footnote{Here and in what
follows we use the convention $\infty^{0}=1.$}: for all $f\in X_{1}\cap
X_{2},$
\begin{equation}
\left\Vert f\right\Vert _{(X_{1},X_{2})_{\theta,q,J}}\leq\left(
(1-\theta)\theta q^{\prime}\right)  ^{-1/q^{\prime}}\left\Vert f\right\Vert
_{X_{1}}^{1-\theta}\left\Vert f\right\Vert _{X_{2}}^{\theta},0<\theta<1,\text{
1}\leq q\leq\infty, \label{constantesj}%
\end{equation}
where $\frac{1}{q}+\frac{1}{q^{\prime}}=1.$

Likewise, for the complex method of interpolation of Calder\'{o}n,
$[.,.]_{\theta},$ we have (cf. \cite{ca}),
\[
\left\Vert f\right\Vert _{[X_{1},.X_{2}]_{\theta}}\leq\left\Vert f\right\Vert
_{X_{1}}^{1-\theta}\left\Vert f\right\Vert _{X_{2}}^{\theta},0<\theta<1,\text{
for }f\in X_{1}\cap X_{2}.
\]
For the \textquotedblleft$K$" method we also have the following result
implicit\footnote{See also \cite{cwja}.} in \cite{jm}, which we prove for sake
of completeness.

\begin{lemma}
\label{marcaada}%
\begin{equation}
\lbrack(1-\theta)\theta q]^{1/q}\left\Vert f\right\Vert _{(X_{1}%
,X_{2})_{\theta,q}}\leq\left\Vert f\right\Vert _{X_{0}}^{1-\theta}\left\Vert
f\right\Vert _{X_{1}}^{\theta},f\in X_{1}\cap X_{2},0<\theta<1. \label{jam1}%
\end{equation}

\end{lemma}

\begin{proof}
Let $f\in X_{1}\cap X_{2}$. Using decompositions of the form $f=f+0,$ or
$f=0+f$, we readily see that
\[
K(f,t;X_{1},X_{2})\leq\min\{\Vert f\Vert_{X_{1}},t\Vert f\Vert_{X_{2}}\}.
\]
Therefore,%
\begin{align*}
\Vert f\Vert_{(X_{1},X_{2})_{\theta,q}}  &  =\left(  \int_{0}^{\infty
}[t^{-\theta}K(f,t;X_{1},X_{2})]^{q}\frac{dt}{t}\right)  ^{1/q}\\
&  \leq\left(  \int_{0}^{\Vert f\Vert_{X_{1}}/\Vert f\Vert_{X_{2}}}%
[t^{-\theta+1}\Vert f\Vert_{X_{2}}]^{q}\frac{dt}{t}+\int_{\Vert f\Vert_{X_{1}%
}/\Vert f\Vert_{X_{2}}}^{\infty}[t^{-\theta}\Vert f\Vert_{X_{1}}]^{q}\frac
{dt}{t}\right)  ^{1/q}\\
&  =\left(  \Vert f\Vert_{X_{2}}^{q}\left(  \frac{\Vert f\Vert_{X_{1}}}{\Vert
f\Vert_{X_{2}}}\right)  ^{q(1-\theta)}\frac{1}{(1-\theta)q}+\Vert
f\Vert_{X_{1}}^{q}\left(  \frac{\Vert f\Vert_{X_{1}}}{\Vert f\Vert_{X_{2}}%
}\right)  ^{-\theta q}\frac{1}{\theta q}\right)  ^{1/q}\\
&  =\Vert f\Vert_{X_{1}}^{1-\theta}\Vert f\Vert_{X_{2}}^{\theta}\left(
\frac{1}{(1-\theta)q}+\frac{1}{\theta q}\right)  ^{1/q}\\
&  =[(1-\theta)\theta q]^{-1/q}\Vert f\Vert_{X_{1}}^{1-\theta}\Vert
f\Vert_{X_{2}}^{\theta}.
\end{align*}

\end{proof}

For more results related to this section we refer to \cite{jm}, \cite{milta}
and \cite{kamixi}.

\section{Self Improving properties of $BMO$ and interpolation\label{secc:self}%
}

The purpose of this section is to provide a new proof of (\ref{intro1}) using
interpolation tools.

One the first results obtained concerning interpolation properties of $BMO$ is
the following (cf. \cite{fest}, \cite{bs} and the references therein)%
\begin{equation}
\left[  L^{1},BMO\right]  _{\theta}=L^{q},\text{ with }\frac{1}{1-\theta}=q.
\label{bmo2}%
\end{equation}
In particular, it follows that%
\[
L^{1}\cap BMO\subset L^{q}.
\]
Therefore, if we work on a cube $Q_{0}$, we have
\[
L^{1}(Q_{0})\cap BMO(Q_{0})=BMO(Q_{0})\subset L^{q}(Q_{0}).
\]
In other words, the following self-improvement holds
\[
f\in BMO(Q_{0})\Rightarrow f\in%
%TCIMACRO{\dbigcap \limits_{q\geq1}}%
%BeginExpansion
{\displaystyle\bigcap\limits_{q\geq1}}
%EndExpansion
L^{q}(Q_{0}).
\]
While it is not true that $f\in BMO(Q_{0})\Rightarrow f\in L^{\infty}(Q_{0}),$
we can quantify precisely the deterioration of the $L^{q}$ norms of a function
in $BMO(Q_{0}),$ to be able to conclude by extrapolation that
\begin{equation}
f\in BMO(Q_{0})\Rightarrow f\in e^{L(Q_{0})}. \label{expo}%
\end{equation}
Let us go over the details. First, consider the following inequality
attributed to Chen-Zhu \cite{chenzu}.

\begin{lemma}
\label{marcao}Let $f\in BMO(Q_{0}),$ and let $1\leq p<\infty.$ Then, there
exists an absolute constant that depends only on $n$ and $p,$ such that for
all $q>p,$%
\begin{equation}
\left\Vert f\right\Vert _{L^{q}}\leq C_{n}q\left\Vert f\right\Vert _{L^{p}%
}^{p/q}\left\Vert f\right\Vert _{BMO}^{1-p/q}. \label{bmo1}%
\end{equation}

\end{lemma}

The point of the result, of course, is the precise dependency of the constants
in terms of $q.$ Before going to the proof let us show how (\ref{expo})
follows from (\ref{bmo1}).

\begin{proof}
\textbf{(of (\ref{expo})}. From (\ref{bmo1}), applied to the case $p=1,$ we
find that, for all $q>1,$%
\begin{align*}
\left\Vert f\right\Vert _{L^{q}}  &  \leq C_{n}q\left\Vert f\right\Vert
_{L^{1}}^{1/q}\left\Vert f\right\Vert _{BMO}^{1-1/q}\\
&  \leq C_{n}q\left\Vert f\right\Vert _{BMO}.
\end{align*}
Hence, using for example the \textquotedblleft$\Delta"$ method of
extrapolation\footnote{For more recent developments in extrapolation theory
cf. \cite{ash}.} of \cite{jm}, we get%
\[
\left\Vert f\right\Vert _{\Delta(\frac{L^{q}}{q})}=\sup_{q>1}\frac{\left\Vert
f\right\Vert _{q}}{q}\approx\left\Vert f\right\Vert _{e^{L}}\leq
c_{n}\left\Vert f\right\Vert _{BMO}.
\]

\end{proof}

We now give a proof of Lemma \ref{marcao} using interpolation.

\begin{proof}
It will be convenient for us to work on $R^{n}$ (the same results hold for
cubes: See more details in Remark \ref{remarkao} below). We start by
considering the case $p=1$, the general case will follow by a re-scaling
argument, which we provide below$.$

The first step is to make explicit the way we obtain the real interpolation
spaces between $L^{1}$ and $BMO$. It is well known (cf. \cite{besa},
\cite{bs}, and the references therein) that
\begin{equation}
(L^{1},BMO)_{1-1/q,q}=L^{q},q>1. \label{interpolabmo}%
\end{equation}
Here the equality of the norms of the indicated spaces is within constants of
equivalence that depend only on $q$ and $n.$ In particular, we have%
\[
\left\Vert f\right\Vert _{L^{q}}\leq c(q,n)\left\Vert f\right\Vert
_{(L^{1},BMO)_{1-1/q,q}}.
\]
The program now is to give a precise estimate of $c(q,n)$ in terms of $q,$ and
then apply Lemma \ref{marcaada}. We shall work with $BMO$ provided by the
seminorm $\left\vert \cdot\right\vert _{BMO}$ (cf. (\ref{llamada}) above).

The following result was proved in \cite[Theorem 6.1]{besa},
\begin{equation}
K(t,f,L^{1}(R^{n}),BMO(R^{n}))\approx tf^{\#\ast}(t), \label{bmo4}%
\end{equation}
with absolute constants of equivalence, and where $f^{\#}$ denotes the sharp
maximal operator\footnote{For computations of related $K-$functionals and
further references cf. \cite{ja}, \cite{alvamil}.} (cf. \cite{fest},
\cite{cosa}, \cite{bs})%
\[
f^{\#}(x):=f_{R^{n}}^{\#}(x)=\sup_{x\ni Q}\frac{1}{\left\vert Q\right\vert
}\int_{Q}\left\vert f(y)-f_{Q}\right\vert dy,\text{ and }f_{Q}=\frac
{1}{\left\vert Q\right\vert }\int_{Q}f(y)dy.
\]
Let $f\in L^{1}\cap BMO,$ $q>1,$ and define $\theta$ by the equation $\frac
{1}{1-\theta}=q.$ Combining (\ref{bmo4}) and (\ref{jam1}), we have, with
absolute constants that do not depend on $q,\theta$ or $f,$%
\begin{align*}
\lbrack(1-\theta)\theta q]^{1/q}\left\{  \int_{0}^{\infty}[t^{-1+1/q}%
tf^{\#\ast}(t)]^{q}\frac{dt}{t}\right\}  ^{1/q}  &  \approx\lbrack
(1-\theta)\theta q]^{1/q}\left\Vert f\right\Vert _{(L^{1}(R^{n}),BMO(R^{n}%
))_{1-1/q,q}}\\
&  \preceq\Vert f\Vert_{L^{1}}^{1/q}\Vert f\Vert_{BMO}^{1-1/q}\text{.}%
\end{align*}
Thus,%
\begin{equation}
\left\{  \int_{0}^{\infty}f^{\#\ast}(t)^{q}dt\right\}  ^{1/q}\leq c(\frac
{q}{q-1})^{1/q}\Vert f\Vert_{L^{1}}^{1/q}\Vert f\Vert_{BMO}^{1-1/q}.
\label{bmo3}%
\end{equation}
Now, we recall that by \cite{bedesa}, as complemented in \cite[(3.8), pag
228]{sasch}, we have
\begin{equation}
f^{\ast\ast}(t)-f^{\ast}(t)\leq cf^{\#\ast}(t),t>0. \label{ufa}%
\end{equation}
Combining (\ref{ufa}) with (\ref{bmo3}), yields
\begin{equation}
\left\{  \int_{0}^{\infty}[f^{\ast\ast}(t)-f^{\ast}(t)]^{q}dt\right\}
^{1/q}\leq c(\frac{q}{q-1})^{1/q}\Vert f\Vert_{L^{1}}^{1/q}\Vert f\Vert
_{BMO}^{1-1/q}. \label{bailada}%
\end{equation}
Observe that, since $[tf^{\ast\ast}(t)]^{\prime}=(\int_{0}^{t}f^{\ast
}(s)ds)^{\prime}=f^{\ast}(t),$ we have
\[
\lbrack-f^{\ast\ast}(t)]^{\prime}=\frac{f^{\ast\ast}(t)-f^{\ast}(t)}{t}.
\]
Moreover, since $f\in L^{1},$ then $f^{\ast\ast}(\infty)=0,$ and it follows
from the fundamental theorem of calculus that we can write%
\[
f^{\ast\ast}(t)=\int_{t}^{\infty}f^{\ast\ast}(s)-f^{\ast}(s)\frac{ds}{s}.
\]
Consequently, by Hardy's inequality,%
\[
\left\{  \int_{0}^{\infty}f^{\ast\ast}(t)^{q}dt\right\}  ^{1/q}\leq q\left\{
\int_{0}^{\infty}[f^{\ast\ast}(t)-f^{\ast}(t)]^{q}dt\right\}  ^{1/q}.
\]
Inserting this information in (\ref{bailada}) we arrive at%
\begin{equation}
\left\{  \int_{0}^{\infty}f^{\ast\ast}(t)^{q}dt\right\}  ^{1/q}\leq
cq(\frac{q}{q-1})^{1/q}\Vert f\Vert_{L^{1}}^{1/q}\Vert f\Vert_{BMO}^{1-1/q}.
\label{verla}%
\end{equation}
We now estimate the left hand side of (\ref{verla}) from below. By the sharp
reverse Hardy inequality for decreasing functions (cf. \cite{ren}, \cite[Lemma
2.1]{mil}, see also \cite{xiao}) we can write%
\begin{align}
\left\Vert f\right\Vert _{q}  &  =\left\{  \int_{0}^{\infty}f^{\ast}%
(t)^{q}dt\right\}  ^{1/q}\nonumber\\
&  \leq(\frac{q-1}{q})^{1/q}\left\{  \int_{0}^{\infty}f^{\ast\ast}%
(t)^{q}dt\right\}  ^{1/q}. \label{gurkha}%
\end{align}
Combining the last inequality with (\ref{verla}) we obtain%
\begin{align*}
\left\Vert f\right\Vert _{q}  &  =\left\{  \int_{0}^{\infty}f^{\ast}%
(t)^{q}dt\right\}  ^{1/q}\\
&  \leq(\frac{q-1}{q})^{1/q}\left\{  \int_{0}^{\infty}f^{\ast\ast}%
(t)^{q}dt\right\}  ^{1/q}\\
&  \leq(\frac{q-1}{q})^{1/q}(\frac{q}{q-1})^{1/q}cq\left\{  \int_{0}^{\infty
}[f^{\ast\ast}(t)-f^{\ast}(t)]^{q}dt\right\}  ^{1/q}\\
&  \leq cq\Vert f\Vert_{L^{1}}^{1/q}\Vert f\Vert_{BMO}^{1-1/q},
\end{align*}
as we wanted to show.

Let us now consider the case $p>1.$ Let $q>p.$ By Holmstedt's reiteration
theorem (cf. \cite{bl}, \cite{holm}) we have%
\[
(L^{p},BMO)_{1-p/q,q}=((L^{1},BMO)_{1-1/p,p},BMO)_{1-p/q},_{q},
\]
and, moreover, with absolute constants that depend only on $p,$
\[
K(t,f;L^{p},BMO)\approx\left\{  \int_{0}^{t^{p}}(s^{\frac{1}{p}-1}sf^{\#\ast
}(s))^{p}\frac{ds}{s}\right\}  ^{1/p}=\left\{  \int_{0}^{t^{p}}f^{\#\ast
}(s)^{p}ds\right\}  ^{1/p}.
\]
By Lemma \ref{marcaada} it follows that, with constants independent of $q,f,$
we have%
\begin{align*}
\left\{  \int_{0}^{\infty}[t^{-(1-p/q)}\left\{  \int_{0}^{t^{p}}f^{\#\ast
}(s)^{p}ds\right\}  ^{1/p}]^{q}\frac{dt}{t}\right\}  ^{1/q}  &  \approx
\left\Vert f\right\Vert _{(L^{p}(R^{n}),BMO(R^{n}))_{1-p/q,q}}\\
&  \preceq p^{-1/q}[q-p]^{-1/q}q^{1/q}\Vert f\Vert_{L^{p}}^{p/q}\Vert
f\Vert_{BMO}^{1-p/q}.
\end{align*}
Now,%
\begin{align*}
\left\{  \int_{0}^{\infty}[t^{-(1-p/q)q}\left\{  \int_{0}^{t^{p}}f^{\#\ast
}(s)^{p}ds\right\}  ^{q/p}\frac{dt}{t}\right\}  ^{1/q}  &  =\left\{  \int
_{0}^{\infty}[t^{-(1-p/q)q}t^{q}\left\{  \frac{1}{t^{p}}\int_{0}^{t^{p}%
}f^{\#\ast}(s)^{p}ds\right\}  ^{q/p}\frac{dt}{t}\right\}  ^{1/q}\\
&  =\left\{  \int_{0}^{\infty}t^{p}\left\{  \frac{1}{t^{p}}\int_{0}^{t^{p}%
}f^{\#\ast}(s)^{p}ds\right\}  ^{q/p}\frac{dt}{t}\right\}  ^{1/q}\\
&  =\left(  \frac{1}{p}\right)  ^{1/q}\left\{  \int_{0}^{\infty}u\left\{
\frac{1}{u}\int_{0}^{u}f^{\#\ast}(s)^{p}ds\right\}  ^{q/p}\frac{du}%
{u}\right\}  ^{1/q}\\
&  =\left(  \frac{1}{p}\right)  ^{1/q}\left[  \left\{  \int_{0}^{\infty
}\left\{  \frac{1}{u}\int_{0}^{u}f^{\#\ast}(s)^{p}ds\right\}  ^{q/p}%
du\right\}  ^{p/q}\right]  ^{1/p}\\
&  \geq\left(  \frac{1}{p}\right)  ^{1/q}\left[  \frac{\frac{q}{p}}%
{q/p-1}\right]  ^{1/q}\left\{  \int_{0}^{\infty}f^{\#\ast}(u)^{q}du\right\}
^{1/q},
\end{align*}
where in the last step we have used the reverse sharp Hardy inequality (cf.
\cite[Lemma 2.1]{mil}). Consequently,%
\begin{align*}
\left\{  \int_{0}^{\infty}f^{\#\ast}(u)^{q}du\right\}  ^{1/q}  &  \preceq
p^{1/q}\left[  \frac{q/p-1}{q/p}\right]  ^{1/q}p^{-1/q}[q-p]^{-1/q}%
q^{1/q}\Vert f\Vert_{L^{p}}^{p/q}\Vert f\Vert_{BMO}^{1-p/q}\\
&  \sim\Vert f\Vert_{L^{p}}^{p/q}\Vert f\Vert_{BMO}^{1-p/q}.
\end{align*}
Hence, by the analysis we already did for the case $p=1,$ we see that%
\begin{align*}
\left\{  \int_{0}^{\infty}f^{\ast}(t)^{q}dt\right\}  ^{1/q}  &  \preceq
(\frac{q-1}{q})^{1/q}q\Vert f\Vert_{L^{p}}^{p/q}\Vert f\Vert_{BMO}^{1-p/q}\\
&  \preceq q\Vert f\Vert_{L^{p}}^{p/q}\Vert f\Vert_{BMO}^{1-p/q},
\end{align*}
as we wished to show.
\end{proof}

\begin{remark}
\label{remarkao}If we work on a cube $Q_{0}$, the replacement of (\ref{ufa})
is (cf. \cite{bs})%
\[
f^{\ast\ast}(t)-f^{\ast}(t)\leq cf^{\#\ast}(t),0<t<\left\vert Q_{0}\right\vert
/3.
\]
In this situation, we have $BMO(Q_{0})\subset L^{1}(Q_{0}),$ and we readily
see that $\left\{  \int_{0}^{\left\vert Q_{0}\right\vert /3}[t^{-1+1/q}%
tf^{\#\ast}(t)]^{q}\frac{dt}{t}\right\}  ^{1/q}$ is an equivalent seminorm for
$(L^{1}(Q_{0}),BMO(Q_{0}))_{1-1/q,q}.$ The rest of the proof now follows
mutatis mutandis.
\end{remark}

\begin{remark}
For related Hardy inequalities for one dimensional oscillation operators of
the form $f_{\#}(t)=$ $\frac{1}{t}\int_{0}^{t}f(s)ds-f(t),$ cf. \cite{misa}
and \cite{krular}.
\end{remark}

\section{The rearrangement Invariant Hull of $BMO$ and Gagliardo coordinate
spaces\label{secc:coordenadas}}

In this section we introduce the \textquotedblleft Gagliardo coordinate
spaces" (cf. \cite{holm}, \cite{jm1}, \cite{mami}) and we use them to extend
(\ref{intro1}) to the setting of real interpolation.

Let $\theta\in\lbrack0,1],q\in(0,\infty].$ Following the discussion given in
the Introduction, we define the \textquotedblleft Gagliardo coordinate spaces"
as follows\footnote{Think in terms of a Gagliardo diagram, see for example
\cite[page 39]{bl}, \cite{jamig}, \cite{jm1}.}%
\[
(X_{1},X_{2})_{\theta,q}^{(1)}=\left\{  f\in X_{1}+X_{2}:\left\Vert
f\right\Vert _{(X_{1},X_{2})_{\theta,q}^{(1)}}<\infty\right\}  ,
\]
where%
\[
\left\Vert f\right\Vert _{(X_{1},X_{2})_{\theta,q}^{(1)}}=\left\{  \int
_{0}^{\infty}\left(  t^{1-\theta}\left[  \frac{K(t,f;X_{1},X_{2})}%
{t}-K^{\prime}(t,f;X_{1},X_{2})\right]  \right)  ^{q}\frac{dt}{t}\right\}
^{1/q},
\]
and%
\[
(X_{1},X_{2})_{\theta,q}^{(2)}=\left\{  f\in X_{1}+X_{2}:\left\Vert
f\right\Vert _{(X_{1},X_{2})_{\theta,q}^{(2)}}<\infty\right\}  ,
\]%
\[
\left\Vert f\right\Vert _{(X_{1},X_{2})_{\theta,q}^{(2)}}=\left\{  \int
_{0}^{\infty}(t^{-\theta}tK^{\prime}(t,f;X_{1},X_{2}))^{q}\frac{dt}%
{t}\right\}  ^{1/q},
\]
and we compare them to the classical Lions-Peetre spaces $(X_{1}%
,X_{2})_{\theta,q}.$

The Gagliardo coordinate spaces in principle are not linear, and the
corresponding functionals, $\left\Vert f\right\Vert _{(X_{1},X_{2})_{\theta
,q}^{(i)}},i=1,2,$ are not norms. However, it turns out that, when $\theta
\in(0,1),q\in(0,\infty],$ we have, with *norm* equivalence (cf. \cite{holm},
\cite{jm1}),
\begin{equation}
(X_{1},X_{2})_{\theta,q}^{(1)}=(X_{1},X_{2})_{\theta,q}^{(2)}=(X_{1}%
,X_{2})_{\theta,q}. \label{equiva}%
\end{equation}
More precisely, the \textquotedblleft norm\textquotedblright\ equivalence
depends only on $\theta,$ and $q.$ On the other hand, at the end points,
$\theta=0$ or $\theta=1,$ the resulting spaces can be very different.

\begin{example}
\label{ejemplomarkao}Let $(X_{1},X_{2})=(L^{1},L^{\infty}).$ Then, if
$\theta=1,q=\infty,$ we have%
\begin{equation}
\left\Vert f\right\Vert _{(X_{1},X_{2})_{1,\infty}^{(1)}}=\left\Vert
f\right\Vert _{L(\infty,\infty)}, \label{wehaveseen}%
\end{equation}
while%
\[
\left\Vert f\right\Vert _{(X_{1},X_{2})_{1,\infty}}=\left\Vert f\right\Vert
_{(X_{1},X_{2})_{1,\infty}^{(2)}}=\left\Vert f\right\Vert _{L^{\infty}}.
\]
For $\theta=1,q<\infty,$%
\[
\left\Vert f\right\Vert _{(X_{1},X_{2})_{1,q}^{(1)}}=\left\Vert f\right\Vert
_{L(\infty,q)}.
\]
On the other hand,%
\[
\left\Vert f\right\Vert _{(X_{1},X_{2})_{1,q}^{(2)}}=\left\{  \int_{0}%
^{\infty}f^{\ast}(t)^{q}\frac{dt}{t}\right\}  ^{1/q}\leq\left\{  \int
_{0}^{\infty}f^{\ast\ast}(t)^{q}\frac{dt}{t}\right\}  ^{1/q}=\left\Vert
f\right\Vert _{(X_{1},X_{2})_{1,q}},
\]
and%
\[
\left\Vert f\right\Vert _{(X_{1},X_{2})_{1,q}^{(2)}}<\infty\Leftrightarrow
f=0.
\]

For $\theta=0,q=\infty,$%
\[
\left\Vert f\right\Vert _{(X_{1},X_{2})_{0,\infty}}=\sup_{t}tf^{\ast\ast
}(t)=\left\Vert f\right\Vert _{L^{1}},
\]
while%
\[
\left\Vert f\right\Vert _{(X_{1},X_{2})_{0,\infty}^{(2)}}=\sup_{t}tf^{\ast
}(t)=\left\Vert f\right\Vert _{L(1,\infty)}.
\]
Moreover,%
\begin{align*}
\left\Vert f\right\Vert _{(X_{1},X_{2})_{0,\infty}^{(1)}}  &  =\sup
_{t}t(f^{\ast\ast}(t)-f^{\ast}(t))\\
&  =\sup_{t}\int_{f^{\ast}(t)}^{\infty}\lambda_{f}(s)ds.
\end{align*}
Therefore, if $f^{\ast}(\infty)=0,$ then%
\begin{align*}
\left\Vert f\right\Vert _{(X_{1},X_{2})_{0,\infty}^{(1)}}  &  =\int
_{0}^{\infty}\lambda_{f}(s)ds\\
&  =\left\Vert f\right\Vert _{L^{1}}.
\end{align*}
Also, if $f^{\ast\ast}(\infty)=0,$%
\begin{align*}
\left\Vert f\right\Vert _{(X_{1},X_{2})_{1,1}^{(1)}}  &  =\int_{0}^{\infty
}\left[  f^{\ast\ast}(t)-f^{\ast}(t)\right]  \frac{dt}{t}\\
&  =\lim_{r\rightarrow0}\int_{r}^{\infty}\left[  f^{\ast\ast}(t)-f^{\ast
}(t)\right]  \frac{dt}{t}\\
&  =\lim_{r\rightarrow0}\left(  f^{\ast\ast}(r)-f^{\ast\ast}(\infty)\right)
\text{ (since }\frac{d}{dr}(-f^{\ast\ast}(t))=\frac{f^{\ast\ast}(t)-f^{\ast
}(t)}{t}\text{)}\\
&  =\left\Vert f\right\Vert _{L^{\infty}}.
\end{align*}

\end{example}

\begin{theorem}
\label{teomarkao}Let $\theta\in\lbrack0,1),$ and let $1\leq q<\infty.$ Then,
there exists an absolute constant $c$ such that%
\begin{equation}
\left\Vert f\right\Vert _{\vec{X}_{\theta,q}^{(2)}}\leq cq\left(
1+[(1-\theta)q]^{1/q}\right)  ^{1-\theta}[(1-\theta)q]^{-\theta}\left\Vert
f\right\Vert _{X_{1}}^{1-\theta}\left\Vert f\right\Vert _{\vec{X}_{1,\infty
}^{(1)}}^{\theta}. \label{verde}%
\end{equation}
In particular, if $(1-\theta)q=1,$%
\begin{equation}
\left\Vert f\right\Vert _{\vec{X}_{\theta,q}^{(2)}}\leq cq\left\Vert
f\right\Vert _{X_{1}}^{1-\theta}\left\Vert f\right\Vert _{\vec{X}_{1,\infty
}^{(1)}}^{\theta}. \label{bmo5}%
\end{equation}

\end{theorem}

Before going to the proof let us argue why such a result could be termed a
generalized John-Nirenberg inequality. Indeed, let $Q$ be a cube in $R^{n},$
and consider the pair $\vec{X}=(L^{1}(Q),L^{\infty}(Q)).$ As we have seen (cf.
(\ref{wehaveseen}))%
\[
\left\Vert f\right\Vert _{\vec{X}_{1,\infty}^{(1)}}=\left\Vert f\right\Vert
_{L(\infty,\infty)}.
\]
By definition, if $f\in L(\infty,\infty)(Q),$ then $f\in L^{1}(Q).$ We now
show that, moreover, $\left\Vert f\right\Vert _{L^{1}}\leq\left\vert
Q\right\vert \left\Vert f\right\Vert _{L(\infty,\infty)}.$ Indeed, for all
$t>0$ we have (cf. \cite[Theorem 2.1 (ii)]{aalto})%
\begin{align}
\int_{\{\left\vert f\right\vert >t\}}\left\vert f(x)\right\vert dx  &
\leq\int_{f^{\ast}(\lambda_{f}(t))}^{\infty}\lambda_{f}(r)dr+t\lambda
_{f}(t)\nonumber\\
&  =\lambda_{f}(t)[f^{\ast\ast}(\lambda_{f}(t))-f^{\ast}(\lambda
_{f}(t))]+t\lambda_{f}(t)\nonumber\\
&  \leq\lambda_{f}(t)\left(  \left\Vert f\right\Vert _{L(\infty,\infty
)}+t\right) \nonumber\\
&  \leq\left\vert Q\right\vert \left(  \left\Vert f\right\Vert _{L(\infty
,\infty)}+t\right)  , \label{nuevo}%
\end{align}
where in the second step we have used the formula (do a graph!)%
\[
s(f^{\ast\ast}(s)-f^{\ast}(s))=\int_{f^{\ast}(s)}^{\infty}\lambda_{f}(u)du.
\]
Let $t\rightarrow0$ in (\ref{nuevo}) then, by Fatou's Lemma, we see that
\begin{equation}
\left\Vert f\right\Vert _{L^{1}}\leq\left\vert Q\right\vert \left\Vert
f\right\Vert _{L(\infty,\infty)}. \label{nuevo1}%
\end{equation}
Now, let $\theta\in(0,1),$ and $\frac{1}{q}=1-\theta.$ Then,%
\begin{align*}
\left\Vert f\right\Vert _{\vec{X}_{\theta,q}^{(2)}}  &  =\left\{  \int
_{0}^{\infty}(t^{-\theta}tK^{\prime}(t,f;L^{1},L^{\infty}))^{q}\frac{dt}%
{t}\right\}  ^{1/q}\\
&  =\left\{  \int_{0}^{\infty}(t^{-(1-1/q)}tf^{\ast}(t))^{q}\frac{dt}%
{t}\right\}  ^{1/q}\\
&  =\left\Vert f\right\Vert _{L^{q}}.
\end{align*}
Therefore, by (\ref{bmo5}) and (\ref{nuevo1}),%
\begin{align*}
\left\Vert f\right\Vert _{L^{q}}  &  \leq cq\left\Vert f\right\Vert _{L^{1}%
}^{1/q}\left\Vert f\right\Vert _{L(\infty,\infty)}^{1/q^{\prime}}\\
&  \leq cq\left\vert Q\right\vert \left\Vert f\right\Vert _{L(\infty,\infty
)}\\
&  \leq cq\left\Vert f\right\Vert _{L(\infty,\infty)}.
\end{align*}
Consequently,%
\[
\left\Vert f\right\Vert _{e^{L}}\approx\left\Vert f\right\Vert _{\Delta
(\frac{L^{q}}{q})}=\sup_{q}\frac{\left\Vert f\right\Vert _{L^{q}}}{q}\leq
c\left\Vert f\right\Vert _{L(\infty,\infty)}.
\]

\begin{proof}
\textbf{(of Theorem (\ref{teomarkao})}. Let us write%
\begin{align*}
\left\Vert f\right\Vert _{\vec{X}_{\theta,q}^{(2)}}  &  =\left(  \int
_{0}^{\infty}(u^{1-\theta}\frac{d}{du}K(u,f;\vec{X}))^{q}\frac{du}{u}\right)
^{1/q}\\
&  =\left(  \int_{0}^{t}(u^{1-\theta}\frac{d}{du}K(u,f;\vec{X}))^{q}\frac
{du}{u}\right)  ^{1/q}+\left(  \int_{t}^{\infty}(u^{1-\theta}\frac{d}%
{du}K(u,f;\vec{X}))^{q}\frac{du}{u}\right)  ^{1/q}\\
&  =(I)+(II).
\end{align*}
We estimate these two terms as follows,%
\begin{align*}
(I)  &  =\left(  \int_{0}^{t}u^{(1-\theta)q}(\frac{d}{du}K(u,f;\vec{X}%
))^{q}\frac{du}{u}\right)  ^{1/q}\\
&  \leq\left(  \int_{0}^{t}u^{(1-\theta)q}(\frac{K(u,f;\vec{X})}{u})^{q}%
\frac{du}{u}\right)  ^{1/q}\text{ (since }\frac{d}{du}K(u,f;\vec{X})\leq
\frac{K(u,f;\vec{X})}{u}).
\end{align*}
On the other hand, since $\left(  \frac{K(u,f;\vec{X})}{u}\right)  ^{\prime
}=\frac{K^{\prime}(u,f;\vec{X})u-K(u,f;\vec{X})}{u^{2}},$ we have that, for
$0<u<t,$%
\begin{align*}
\frac{K(u,f;\vec{X})}{u}  &  =\frac{K(t,f;\vec{X})}{t}+\left(  -\left.
\frac{K(\cdot,f;\vec{X})}{\cdot}\right\vert _{u}^{t}\right) \\
&  =\frac{K(t,f;\vec{X})}{t}+\int_{u}^{t}\left(  \frac{K(r,f;\vec{X})}%
{r}-K^{\prime}(r,f;\vec{X})\right)  \frac{dr}{r}\\
&  \leq\frac{K(t,f;\vec{X})}{t}+\left(  \log\frac{t}{u}\right)  \sup_{r\leq
t}\left(  \frac{K(r,f;\vec{X})}{r}-K^{\prime}(r,f;\vec{X})\right) \\
&  \leq\frac{K(t,f;\vec{X})}{t}+\log\frac{t}{u}\left\Vert f\right\Vert
_{\vec{X}_{1,\infty}^{(1)}}.
\end{align*}
Therefore, by the triangle inequality,%
\begin{align*}
(I)  &  \leq\left(  \int_{0}^{t}(u^{(1-\theta)q}\{\frac{K(t,f;\vec{X})}%
{t}+\log\frac{t}{u}\left\Vert f\right\Vert _{\vec{X}_{1,\infty}^{(1)}}%
\}^{q}\frac{du}{u}\right)  ^{1/q}\\
&  \leq\frac{K(t,f;\vec{X})}{t}\left(  \int_{0}^{t}u^{(1-\theta)q}\frac{du}%
{u}\right)  ^{1/q}+\left\Vert f\right\Vert _{\vec{X}_{1,\infty}^{(1)}}\left(
\int_{0}^{t}u^{(1-\theta)q}\left(  \log\frac{t}{u}\right)  ^{q}\frac{du}%
{u}\right)  ^{1/q}\\
&  =\frac{K(t,f;\vec{X})}{t}\frac{t^{(1-\theta)}}{[(1-\theta)q]^{1/q}%
}+\left\Vert f\right\Vert _{\vec{X}_{1,\infty}^{(1)}}\left(  \int_{0}%
^{t}u^{(1-\theta)q}\left(  \log\frac{t}{u}\right)  ^{q}\frac{du}{u}\right)
^{1/q}\\
&  =(a)+(b).
\end{align*}
For the term $(a)$ we have%
\begin{align*}
(a)  &  \leq\frac{t^{-\theta}}{[(1-\theta)q]^{1/q}}\lim_{t\rightarrow\infty
}K(t,f;\vec{X})\text{ (since }K(\cdot,f;\vec{X})\text{ increases)}\\
&  \leq\frac{t^{-\theta}}{[(1-\theta)q]^{1/q}}\left\Vert f\right\Vert _{X_{1}%
}.
\end{align*}
To deal with $(b)$ we use the asymptotics of the gamma function as follows:
let $s=\log\frac{t}{u},$ then $u=te^{-s},$ $du=-te^{-s}ds,$ $\frac{du}%
{u}=-ds,u^{(1-\theta)q}=t^{(1-\theta)q}e^{-s(1-\theta)q},$ and we have%
\begin{align*}
(b)  &  =\left\Vert f\right\Vert _{\vec{X}_{1,\infty}^{(1)}}\left(  \int
_{0}^{t}u^{(1-\theta)q}s^{q}\frac{du}{u}\right)  ^{1/q}\\
&  =\left\Vert f\right\Vert _{\vec{X}_{1,\infty}^{(1)}}t^{1-\theta}\left(
\int_{0}^{\infty}e^{-s(1-\theta)q}s^{q}ds\right)  ^{1/q}\\
&  =\left\Vert f\right\Vert _{\vec{X}_{1,\infty}^{(1)}}t^{1-\theta}\left(
\int_{0}^{\infty}e^{-\tau}\frac{\tau^{q}}{[(1-\theta)q]^{q}}\frac{d\tau
}{[(1-\theta)q]}\right)  ^{1/q},\text{ (let }\tau=s(1-\theta)q)\\
&  =\frac{\left\Vert f\right\Vert _{\vec{X}_{1,\infty}^{(1)}}}{[(1-\theta
)q]}\frac{1}{[(1-\theta)q]^{1/q}}t^{1-\theta}\left(  \Gamma(q+1)\right)
^{1/q}\\
&  \leq\frac{\left\Vert f\right\Vert _{\vec{X}_{1,\infty}^{(1)}}}%
{[(1-\theta)q]}\frac{1}{[(1-\theta)q]^{1/q}}t^{1-\theta}q\text{.}%
\end{align*}
Combining inequalities for $(a)$ and $(b)$ we have%
\[
(I)\leq\frac{t^{-\theta}}{[(1-\theta)q]^{1/q}}\left\Vert f\right\Vert _{X_{1}%
}+\frac{\left\Vert f\right\Vert _{\vec{X}_{1,\infty}^{(1)}}}{[(1-\theta
)q]}\frac{1}{[(1-\theta)q]^{1/q}}t^{1-\theta}q.
\]
We now estimate $(II):$%
\begin{align*}
(II)  &  =\left(  \int_{t}^{\infty}u^{(1-\theta)q}u^{-1}\left(  \frac{d}%
{du}K(u,f;\vec{X})\right)  ^{q-1}\left(  u\frac{d}{du}K(u,f;\vec{X})\right)
\frac{du}{u}\right)  ^{1/q}\\
&  \leq\left\{  \sup_{u\geq t}(u^{\frac{(1-\theta)q-1}{q}}\left(  \frac{d}%
{du}K(u,f;\vec{X})\right)  ^{\frac{q-1}{q}})\right\}  \left\{  \int
_{t}^{\infty}\frac{d}{du}K(u,f;\vec{X})du\right\}  ^{1/q}\\
&  =(c)(d).
\end{align*}
The factors on the right hand side can be estimated as follows,
\begin{align*}
(d)  &  =\left(  \lim_{u\mapsto\infty}K(u,f;\vec{X})-K(t,f;\vec{X})\right)
^{1/q}\\
&  \leq\left(  \lim_{u\mapsto\infty}K(u,f;\vec{X})\right)  ^{1/q}\\
&  =\left\Vert f\right\Vert _{X_{0}}^{1/q}.
\end{align*}
Also, since $K(\cdot,f;\vec{X})$ is concave, $\frac{d}{du}K(u,f;\vec{X}%
)\leq\frac{K(u,f;\vec{X})}{u},$ consequently,%
\begin{align*}
(c)  &  \leq\left\Vert f\right\Vert _{X_{1}}^{1-1/q}\sup_{u\geq t}%
\{u^{\frac{(1-\theta)q-1}{q}-\frac{q-1}{q}}\}\\
&  \leq\left\Vert f\right\Vert _{X_{1}}^{1-1/q}\{\sup_{u\geq t}u^{-\theta}\}\\
&  =\left\Vert f\right\Vert _{X_{1}}^{1-1/q}t^{-\theta}.
\end{align*}
Thus,%
\begin{align*}
(II)  &  \leq\left\Vert f\right\Vert _{X_{1}}^{1/q}\left\Vert f\right\Vert
_{X_{1}}^{1-1/q}t^{-\theta}\\
&  =\left\Vert f\right\Vert _{X_{1}}t^{-\theta}.
\end{align*}
Combining the estimates for $(I)$ and $(II)$ yields,%
\begin{align}
\left\Vert f\right\Vert _{\vec{X}_{\theta,q}^{(2)}}  &  \leq\left(
\frac{1+[(1-\theta)q]^{1/q}}{[(1-\theta)q]^{1/q}}\right)  t^{-\theta
}\left\Vert f\right\Vert _{X_{1}}+\frac{1}{[(1-\theta)q]}\frac{1}%
{[(1-\theta)q]^{1/q}}qt^{1-\theta}\left\Vert f\right\Vert _{\vec{X}_{1,\infty
}^{(1)}}\nonumber\\
&  \leq cq\frac{1}{[(1-\theta)q]^{1/q}}\left\{  \left(  1+[(1-\theta
)q]^{1/q}\right)  t^{-\theta}\left\Vert f\right\Vert _{X_{1}}+\frac
{1}{[(1-\theta)q]}t^{1-\theta}\left\Vert f\right\Vert _{\vec{X}_{1,\infty
}^{(1)}}\right\}  . \label{dolores}%
\end{align}
We balance the terms on the right hand side by choosing $t$ such that%
\[
\left(  1+[(1-\theta)q]^{1/q}\right)  t^{-\theta}\left\Vert f\right\Vert
_{X_{1}}=\frac{1}{[(1-\theta)q]}t^{1-\theta}\left\Vert f\right\Vert _{\vec
{X}_{1,\infty}^{(1)}},
\]
whence,%
\[
t=\left(  1+[(1-\theta)q]^{1/q}\right)  [(1-\theta)q]\frac{\left\Vert
f\right\Vert _{X_{1}}}{\left\Vert f\right\Vert _{\vec{X}_{1,\infty}^{(1)}}}.
\]
Inserting this value of $t$ in (\ref{dolores}) we find%
\begin{align*}
\left\Vert f\right\Vert _{\vec{X}_{\theta,q}^{(2)}}  &  \leq cq\frac
{1}{[(1-\theta)q]^{1/q}}\left\{  \left(  1+[(1-\theta)q]^{1/q}\right)
^{-\theta}[(1-\theta)q]^{-\theta}\left\Vert f\right\Vert _{X_{1}}^{1-\theta
}\left\Vert f\right\Vert _{\vec{X}_{1,\infty}^{(1)}}^{\theta}\right\} \\
&  \leq cq\left(  \frac{1}{[(1-\theta)q]^{1/q}}\right)  \left(  1+[(1-\theta
)q]^{1/q}\right)  ^{-\theta}[(1-\theta)q]^{-\theta}\left\Vert f\right\Vert
_{X_{1}}^{1-\theta}\left\Vert f\right\Vert _{\vec{X}_{1,\infty}^{(1)}}%
^{\theta},
\end{align*}
as we wished to show
\end{proof}

\begin{remark}
As an easy application of (\ref{bmo5}), we note that, if $X_{2}\subset X_{1},$
we can write%
\[
\left\Vert f\right\Vert _{\Delta\left(  (1-\theta)\vec{X}_{\theta,\frac
{1}{1-\theta}}^{(2)}\right)  }=\sup_{\theta}(1-\theta)\left\Vert f\right\Vert
_{\vec{X}_{\theta,\frac{1}{1-\theta}}^{(2)}}\leq c\left\Vert f\right\Vert
_{\vec{X}_{1,\infty}^{(1)}}.
\]
We believe that similar arguments would lead to computations with the
$\Delta_{p}$ method of extrapolation (cf. \cite{jm}, \cite{kami}) but this
lies outside the scope of this paper so we leave the issue for another occasion.
\end{remark}

\section{Recent uses of the oscillation operator and $L(\infty,q)$ spaces in
Analysis\label{secc:uses}}

We present several different applications connected with the material
developed in this note. The material is only but a sample of results. The
results presented are either new or they provide a new treatment to known
results\footnote{See also \cite{rota}, Lesson \#3.}. This section differs from
previous ones in that we proceed formally and, whenever possible, we refer the
reader to the literature for background material and complete details. Further
development of materials in this section will appear elsewhere, e.g. in
\cite{cwbjmm}, \cite{mil1}, \cite{mil2},..)

\subsection{On some inequalities for classical operators by
Bennett-DeVore-Sharpley and Bagby and Kurtz\label{secc:singular}}

In this section we show how the methods developed in this paper can be applied
to give a new approach to results on singular integrals and maximal operators
that appeared first in \cite{bedesa}, \cite{bagby}, and \cite{kurtz} (cf. also
the references therein).

Let $\not T  $ and $U$ be operators acting in a sufficiently large class of
testing functions, say the space $S$ of Schwartz testing functions on $R^{n}$.
Furthermore, suppose that there exists $C>0,$ such that for all $f\in S,$ the
following pointwise inequality holds%
\[
\left(  Tf\right)  ^{\#}(x)\leq CUf(x).
\]
Then, taking rearrangements we have%
\[
\left(  Tf\right)  ^{\#\ast}(t)\leq C(Uf)^{\ast}(t),t>0.
\]
Therefore,%
\[
\left\{  \int_{0}^{\infty}\left(  Tf\right)  ^{\#\ast}(t)^{p}dt\right\}
^{1/p}\leq C\left\{  \int_{0}^{\infty}(Uf)^{\ast}(t)^{p}dt\right\}  ^{1/p}.
\]
Now, by (\ref{ufa}) above, and the analysis that follows it, we see that%

\begin{align*}
\left\{  \int_{0}^{\infty}\left(  Tf\right)  ^{\ast}(t)^{p}dt\right\}  ^{1/p}
&  \leq\left(  \frac{p-1}{p}\right)  ^{1/p}pC\left\{  \int_{0}^{\infty
}(Uf)^{\ast}(t)^{p}dt\right\}  ^{1/p}\\
&  \leq Cp\left\{  \int_{0}^{\infty}(Uf)^{\ast}(t)^{p}dt\right\}
^{1/p}\text{.}%
\end{align*}
In other words,%
\begin{equation}
\left\Vert Tf\right\Vert _{p}\leq cp\left\Vert Uf\right\Vert _{p},
\label{ladej}%
\end{equation}
and we recover the main result of \cite{kurtz}.

We should also point out that the method of proof can be also implemented to
deal with the corresponding more general inequalities for doubling measures on
$R^{n}$ (cf. \cite{cof}, \cite{kurtz}, and the references therein).

\begin{remark}
It may be appropriate to mention that once one knows (\ref{ladej}) then one
could use the extrapolation theory of \cite{jm} to show (cf. \cite[Lemma
5]{kurtz}) that there exist absolute constants $C,\gamma>0,$ such that,%
\begin{equation}
(Tf)^{\ast}(t)\leq C\int_{\gamma t}^{\infty}\left(  Uf\right)  ^{\ast}%
(s)\frac{ds}{s},\text{ for all }f\in S. \label{ladej1}%
\end{equation}

\end{remark}

\subsection{Good-Lambda Inequalities\label{secc:lambda}}

These inequalities apparently originate in the celebrated work of
Burkholder-Gundy \cite{bg} (cf. also \cite{bu}) on extrapolation of martingale
inequalities. They have been used since then to great effect in probability,
and also in classical harmonic analysis, probably beginning with \cite{bg1}
and Coifman-Fefferman \cite{cof}. Inequalities on the oscillation operator
$f^{\ast\ast}-f^{\ast}$ are closely connected with good-lambda inequalities.
This connection was pointed out long ago by Neveu \cite{nev}, Herz
\cite{herz}, Bagby and Kurtz (cf. \cite{bagby}, \cite{kurtz}), among others.
In this section we formalize some of their ideas.

To fix matters, let $\mu$ be a measure on $R^{n},$ and let $T$ and $H$ be
operators acting on a sufficiently rich class of functions. A prototypical
good lambda inequality has the following form: for all $\lambda>0,\varepsilon
>0,$ there exists $c(\varepsilon)>0,$ with $c(\varepsilon)\rightarrow0,$ as
$\varepsilon\rightarrow0,$ such that
\begin{equation}
\mu\{\left\vert Tf\right\vert >2\lambda,\left\vert Hf\right\vert
\leq\varepsilon\lambda\}\leq c(\varepsilon)\mu\{\left\vert Tf\right\vert
>\lambda\}. \label{vale0}%
\end{equation}
The idea here is that if the behavior of $H$ is known on r.i. spaces, say on
$L^{p}$ spaces$,$ then we can also control the behavior of $T.$ Indeed, the
distribution function of $Tf$ can be controlled by the following elementary
argument%
\begin{align*}
\mu\{\left\vert Tf\right\vert  &  >2\lambda\}\leq\mu\{\left\vert Tf\right\vert
>2\lambda,\left\vert Hf\right\vert \leq\varepsilon\lambda\}+\mu\{\left\vert
Tf\right\vert >2\lambda,\left\vert Hf\right\vert >\varepsilon\lambda\}\\
&  \leq c(\varepsilon)\mu\{\left\vert Tf\right\vert >\lambda\}+\mu\{\left\vert
Hf\right\vert >\varepsilon\lambda\}.
\end{align*}
Then, since%
\[
\left\Vert f\right\Vert _{p}^{p}=p\int_{0}^{\infty}\lambda^{p-1}%
\mu\{\left\vert f\right\vert >\lambda\}d\lambda,
\]
we readily see that we can estimate the norm of $\left\Vert Tf\right\Vert
_{p}^{p}$ in terms of the norm of $\left\Vert Hf\right\Vert _{p}^{p}$ by means
of making $\varepsilon$ sufficiently small in order to be able collect the two
$\left\Vert Tf\right\Vert _{p}^{p}$ terms on the left hand side of the inequality.

In \cite{kurtz}, the author shows the following stronger good lambda
inequality for $f^{\#}:$There exists $B>0,$ such that for all $\varepsilon
>0,\lambda>0,$ and all locally integrable $f,$ we have%
\[
\mu\{\left\vert f\right\vert >Bf^{\#}+\lambda\}\leq\varepsilon\mu\{\left\vert
f\right\vert >\lambda\}.
\]
This inequality is used to show the following oscillation inequality (cf.
\cite[p 270]{kurtz})
\[
f^{\ast}(t)-f^{\ast}(2t)\leq Cf^{\#\ast}(\frac{t}{2}),t>0,
\]
where $C$ is an absolute constant.

More generally, the argument in \cite{kurtz} can be formalized as follows

\begin{theorem}
\label{hercules}Suppose that $T$ and $H$ are operators acting on the Schwartz
class $S,$ such that, moreover, for all $\varepsilon>0,$ there exists $B>0,$
such that for all $\lambda>0,$%
\begin{equation}
\mu\{\left\vert Tf\right\vert >B\left\vert Hf\right\vert +\lambda
\}\leq\varepsilon\mu\{\left\vert Tf\right\vert >\lambda\}. \label{vale}%
\end{equation}
Then, there exists a constant $C>0$ such that for all $t>0,$ and for all $f\in
S,$%
\begin{equation}
\left(  Tf\right)  ^{\ast}(t)-\left(  Tf\right)  ^{\ast}(2t)\leq C\left(
Hf\right)  ^{\ast}(\frac{t}{2}). \label{valetodo}%
\end{equation}

\end{theorem}

\begin{proof}
Let $\varepsilon=\frac{1}{4},$ and fix $B:=B(\frac{1}{4})$ such that
(\ref{vale}) holds for all $\lambda>0.$ Let $f\in S,$ and select
$\lambda=\left(  Tf\right)  ^{\ast}(2t).$ Then,%
\[
\mu\{\left\vert Tf\right\vert >B\left\vert Hf\right\vert +\left(  Tf\right)
^{\ast}(2t)\}\leq\frac{1}{4}\mu\{\left\vert Tf\right\vert >\left(  Tf\right)
^{\ast}(2t)\}\leq\frac{t}{2}.
\]

By definition we have,%
\[
\mu\{\left\vert Hf\right\vert >\left(  Hf\right)  ^{\ast}(\frac{t}{2}%
)\}\leq\frac{t}{2}.
\]
Consider the set $A=\{\left\vert Tf\right\vert >B(Hf)^{\ast}(\frac{t}%
{2})+\left(  Tf\right)  ^{\ast}(2t)\}.$ Then, it is easy to see, by
contradiction, that%
\[
A\subset\{\left\vert Tf\right\vert >B\left\vert Hf\right\vert +\left(
Tf\right)  ^{\ast}(2t)\}%
%TCIMACRO{\dbigcup }%
%BeginExpansion
{\displaystyle\bigcup}
%EndExpansion
\{\left\vert Hf\right\vert >\left(  Hf\right)  ^{\ast}(t/2)\}.
\]
Consequently,%
\[
\mu(A)\leq\frac{t}{2}+\frac{t}{2}.
\]
Now, since%
\[
\left(  Tf\right)  ^{\ast}(t)=\inf\{s:\mu\{\left\vert Tf\right\vert >s\}\leq
t\},
\]
it follows that%
\[
\left(  Tf\right)  ^{\ast}(t)\leq B(Hf)^{\ast}(\frac{t}{2})+\left(  Tf\right)
^{\ast}(2t),
\]
as we wished to show.
\end{proof}

\begin{remark}
It is easy to compare the oscillation operators $\left(  Tf\right)  ^{\ast
}(t)-\left(  Tf\right)  ^{\ast}(2t)$ and $\left(  Tf\right)  ^{\ast\ast
}(t)-\left(  Tf\right)  ^{\ast}(t).$ For example, it is shown in \cite[Theorem
4.1 p 1223]{bmr} that%
\[
\left(  Tf\right)  ^{\ast}(\frac{t}{2})-\left(  Tf\right)  ^{\ast}%
(t)\leq2\left(  \left(  Tf\right)  ^{\ast\ast}(t)-\left(  Tf\right)  ^{\ast
}(t)\right)  ,
\]
and%
\begin{align}
\left(  \left(  Tf\right)  ^{\ast\ast}(t)-\left(  Tf\right)  ^{\ast
}(t)\right)   &  \leq\frac{1}{t}\int_{0}^{t}\left(  \left(  Tf\right)  ^{\ast
}(\frac{s}{2})-\left(  Tf\right)  ^{\ast}(s)\right)  ds\nonumber\\
&  +\left(  Tf\right)  ^{\ast}(\frac{t}{2})-\left(  Tf\right)  ^{\ast}(t).
\label{combinaos}%
\end{align}

\end{remark}

Combining Theorem \ref{hercules} and the previous remark we have the following

\begin{theorem}
Suppose that $T$ and $H$ satisfy the strong good-lambda inequality
(\ref{vale}). Then,%
\[
\left(  \left(  Tf\right)  ^{\ast\ast}(t)-\left(  Tf\right)  ^{\ast
}(t)\right)  \leq2B(Hf)^{\ast\ast}(\frac{t}{4}).
\]

\end{theorem}

\begin{proof}
The desired result follows combining (\ref{valetodo}) with (\ref{combinaos}).
\end{proof}

\begin{remark}
It is easy to convince oneself that the good-lambda inequalities of the form
(\ref{vale}) are, in fact, stronger than the usual good-lambda inequalities,
e.g. of the form (\ref{vale0}) (cf. \cite{kurtz}).
\end{remark}

\begin{remark}
Clearly there are many nice results lurking in the background of this section.
For example, a topic that comes to mind is to explore the use of good-lambda
inequalities in the interpolation theory of operator ideals and its
applications (cf. \cite{cobos}, \cite{mastylo}, and the references therein).
On the classical analysis side it would be of interest to explore the
connections of the interpolation methods with the maximal inequalities due to
Muckenhoupt-Wheeden and Hedberg-Wolff (cf. \cite{arz}, \cite{hon} and the
references therein).
\end{remark}

\subsection{Extrapolation of inequalities: Burkholder-Gundy-Herz meet
Calder\'{o}n-Maz'ya and Cwikel et al.\label{secc:bgextra}}

The leitmotif of \cite{bg} is the extrapolation of inequalities for the
classical operators acting on martingales (e.g. martingale transforms, maximal
operators, square functions..). There are two main ingredients to the method.
First, the authors, modeling on the classical operators acting on martingales,
single out properties that the operators under consideration will be required
to satisfy. Then they usually assume that an $L^{p}$ or weak type $L^{p}$
estimate holds, and from this information they deduce a full family of $L^{p}$
or even Orlicz inequalities. The main technical step of the extrapolation
procedure consists of using the assumptions we have just described in order to
prove suitable good-lambda inequalities. The method is thus different from the
usual interpolation theory, which works for \textbf{all} operators that
satisfy a \textbf{pair} of given estimates.

In \cite{cwbjmm} we have shown how to formulate some of the assumptions of
\cite{bg} in terms of optimal decompositions to compute $K-$functionals. Then,
assuming that the operators to be extrapolated act on interpolation spaces,
one can extract oscillation inequalities using interpolation theory. In
particular, the developments in \cite{cwbjmm} allow the extrapolation of
operators that do not necessarily act on martingales, but also on function
spaces, e.g. gradients, square functions, Littlewood-Paley functions, etc. The
basic technique involved to achieve the extrapolation is to use the
assumptions to prove an oscillation rearrangement inequality\footnote{Herz
\cite{he} also developed a different technique to extrapolate oscillation
rearrangement inequalities for martingale operators.}.

Unfortunately, \cite{cwbjmm} is still unpublished, although some of the
results have been discussed elsewhere (cf. \cite{mamicoc}) or will appear soon
(cf. \cite{mil1}). In keeping with the theme of this note, in this section I
want to present some more details on how one can extrapolate Sobolev
inequalities and encode the information using the oscillation operator
$f^{\ast\ast}-f^{\ast}.$

Let us take as a starting point the weak type Gagliardo-Nirenberg Sobolev
inequality in $R^{n}$ (cf. \cite{leo})$:$%
\begin{equation}
\left\Vert f\right\Vert _{L(n^{\prime},\infty)}=\left\Vert f\right\Vert
_{(L^{1}(R^{n}),L^{\infty}(R^{n}))_{1/n,\infty}}\leq c_{n}\left\Vert \nabla
f\right\Vert _{L^{1}},f\in Lip_{0}(R^{n}). \label{abarca}%
\end{equation}
Let $f\in Lip_{0}(R^{n}),$ and assume without loss that $f$ is positive. Let
$t>0,$ then an optimal decomposition for the computation of
\[
K(t,f):=K(t,f;L^{1}(R^{n}),L^{\infty}(R^{n}))=\int_{0}^{t}f^{\ast}(s)ds,
\]
is given by%
\begin{equation}
f=f_{f^{\ast}(t)}+(f-f_{f^{\ast}(t)}), \label{shows}%
\end{equation}
where%
\begin{equation}
f_{f^{\ast}(t)}(x)=\left\{
\begin{array}
[c]{ll}%
f(x)-f^{\ast}(t) & \text{if }f^{\ast}(t)<f(x)\\
0 & \text{if }f(x)\leq f^{\ast}(t)
\end{array}
\right.  . \label{rf3}%
\end{equation}
By direct computation we have%
\begin{align*}
K(t,f)  &  \leq\left\Vert f_{f^{\ast}(t)}\right\Vert _{L^{1}}+t\left\Vert
f-f_{f^{\ast}(t)}\right\Vert _{L^{\infty}}\\
&  =\left(  \int_{0}^{t}f^{\ast}(s)ds-tf^{\ast}(t)\right)  +tf^{\ast}(t)\\
&  =\int_{0}^{t}f^{\ast}(s)ds.
\end{align*}
On the other hand, if $f=f_{0}+f_{1},$ with $f_{0}\in L^{1},f_{1}\in
L^{\infty},$ then%
\begin{align*}
\int_{0}^{t}f^{\ast}(s)ds  &  \leq\int_{0}^{t}f_{0}^{\ast}(s)ds+\int_{0}%
^{t}f_{1}^{\ast}(s)ds\\
&  \leq\left\Vert f_{0}\right\Vert _{L^{1}}+t\left\Vert f_{1}\right\Vert
_{L^{\infty}}.
\end{align*}
Therefore,%
\[
K(t,f)=\left\Vert f_{f^{\ast}(t)}\right\Vert _{L^{1}}+t\left\Vert
f-f_{f^{\ast}(t)}\right\Vert _{L^{\infty}}.
\]
Also note that, confirming (\ref{a1}) and (\ref{a2}) above, by direct
computation we have,%
\begin{align*}
\left\Vert f_{f^{\ast}(t)}\right\Vert _{L^{1}}  &  =\int_{0}^{t}f^{\ast
}(s)ds-tf^{\ast}(t)\\
&  =t\left(  f^{\ast\ast}(t)-f^{\ast}(t)\right)  ,
\end{align*}%
\[
\left\Vert f-f_{f^{\ast}(t)}\right\Vert _{L^{\infty}}=f^{\ast}(t).
\]
The commutation of the gradient with truncations\footnote{(cf. \cite{maz},
\cite{bakr}, \cite{haj}, \cite{mmp})} implies%
\[
\left\Vert \nabla f_{f^{\ast}(t)}\right\Vert _{L^{1}}\leq\int_{\{f>f^{\ast
}(t)\}}\left\vert \nabla f\right\vert dx.
\]
Therefore%
\[
\left\Vert \nabla f_{f^{\ast}(t)}\right\Vert _{L^{1}}\leq\int_{0}%
^{t}\left\vert \nabla f\right\vert ^{\ast}(s)ds.
\]
We apply the inequality (\ref{abarca}) to $f_{f^{\ast}(t)}$. We find%
\begin{align*}
\left\Vert f_{f^{\ast}(t)}\right\Vert _{(L^{1}(R^{n}),L^{\infty}%
(R^{n}))_{1/n,\infty}}  &  \leq c_{n}\left\Vert \nabla f_{f^{\ast}%
(t)}\right\Vert _{L^{1}}\\
&  \leq c_{n}\int_{0}^{t}\left\vert \nabla f\right\vert ^{\ast}(s)ds.
\end{align*}
We estimate the left hand side%
\begin{align*}
\left\Vert f_{f^{\ast}(t)}\right\Vert _{(L^{1}(R^{n}),L^{\infty}%
(R^{n}))_{1/n,\infty}}  &  =\sup_{s>0}s^{-1/n}K(s,f_{f^{\ast}(t)})\\
&  =\sup_{s>0}s^{-1/n}K(s,f-(f-f_{f^{\ast}(t)}))\\
&  \geq t^{-1/n}K(t,f-(f-f_{f^{\ast}(t)}))\\
&  \geq t^{-1/n}\{K(t,f)-K(t,f-f_{f^{\ast}(t)})\},
\end{align*}
where the last inequality follows by the triangle inequality, since
$K(t,\cdot)$ is a norm. Now,%
\[
K(t,f)=tf^{\ast\ast}(t),
\]
and%
\begin{align*}
K(t,f-f_{f^{\ast}(t)})  &  \leq t\left\Vert f_{f^{\ast}(t)}-f\right\Vert
_{L^{\infty}}\\
&  =tf^{\ast}(t).
\end{align*}
Thus,%
\[
\left\Vert f_{f^{\ast}(t)}\right\Vert _{(L^{1}(R^{n}),L^{\infty}%
(R^{n}))_{1/n,\infty}}\geq t^{-1/n}\left(  tf^{\ast\ast}(t)-tf^{\ast
}(t)\right)  .
\]
Combining estimates we find%
\[
\left(  tf^{\ast\ast}(t)-tf^{\ast}(t)\right)  t^{-1/n}\leq c_{n}\int_{0}%
^{t}\left\vert \nabla f\right\vert ^{\ast}(s)ds,
\]
which can be written as%
\begin{equation}
f^{\ast\ast}(t)-f^{\ast}(t)\leq c_{n}t^{1/n}\left\vert \nabla f\right\vert
^{\ast\ast}(t). \label{venta}%
\end{equation}
This inequality had been essentially obtained by Kolyada \cite{kol} and is
equivalent (cf. \cite{mmp}) to earlier inequalities by Talenti \cite{talenti}%
.. but its role in the study of limiting Sobolev inequalities, and the
introduction of the $L(\infty,q)$ spaces, was only pointed out in \cite{bmr}.
In \cite{mmp} it was shown that (\ref{venta}) is equivalent to the
isoperimetric inequality. More generally, Martin-Milman extended (\ref{venta})
to Gaussian measures (\cite{mamijfa}), and later\footnote{See also
\cite{kalis} for a maximal function approach to oscillation inequalities for
the gradient.} (cf. \cite{mamica}) to metric measure spaces, where the
inequality takes the following form%
\begin{equation}
f^{\ast\ast}(t)-f^{\ast}(t)\leq\frac{t}{I(t)}\left\vert \nabla f\right\vert
^{\ast\ast}(t), \label{venta1}%
\end{equation}
where $I$ is the isoperimetric profile associated with the underlying
geometry. In fact they show the equivalence of (\ref{venta1}) with the
corresponding isoperimetric inequality (for a recent survey cf. \cite{mamica}).

One can prove the Gaussian Sobolev version of (\ref{venta1}) using the same
extrapolation procedure as above, but using as a starting point Ledoux's
inequality \cite{le} as a replacement of the Gagliardo-Nirenberg inequality
(cf. \cite{mamicoc}). Further recent extensions of the Martin-Milman
inequality on Gaussian measure can be found in \cite{xiao1}.

Let us now show another Sobolev rearrangement inequality for oscillations,
apparently first recorded in \cite{mmp}. We now take as our starting
inequality for the extrapolation procedure the sharp form of the
Gagliardo-Nirenberg inequality, which can be formulated as%
\[
\left\Vert f\right\Vert _{L(n^{\prime},1)}=\left\Vert f\right\Vert
_{(L^{1}(R^{n}),L^{\infty}(R^{n}))_{1/n,1}}\leq c_{n}\left\Vert \nabla
f\right\Vert _{L^{1}},f\in Lip_{0}(R^{n}).
\]
We apply the inequality to $f_{f^{\ast}(t)}.$ The right hand side we have
already estimated,%
\[
\left\Vert f_{f^{\ast}(t)}\right\Vert _{(L^{1}(R^{n}),L^{\infty}%
(R^{n}))_{1/n,1}}\leq c_{n}\int_{0}^{t}\left\vert \nabla f\right\vert ^{\ast
}(s)ds.
\]
Now,%
\begin{align*}
\left\Vert f_{f^{\ast}(t)}\right\Vert _{(L^{1}(R^{n}),L^{\infty}%
(R^{n}))_{1/n,1}} &  =\int_{0}^{\infty}s^{-1/n}K(s,f_{f^{\ast}(t)})\frac
{ds}{s}\\
&  \geq\int_{0}^{t}s^{-1/n}K(s,f-(f-f_{f^{\ast}(t)}))\frac{ds}{s}\\
&  \geq\int_{0}^{t}s^{-1/n}\{K(s,f)-K(s,f-f_{f^{\ast}(t)})\}\frac{ds}{s}\\
&  =\int_{0}^{t}s^{-1/n}\{sf^{\ast\ast}(s)-K(s,f-f_{f^{\ast}(t)})\}\frac
{ds}{s}.
\end{align*}
Note that since $f^{\ast}$ decreases, for $s\leq t,$ we have%
\begin{align*}
K(s,f-f_{f^{\ast}(t)}) &  \leq s\left\Vert f-f_{f^{\ast}(t)}\right\Vert
_{L^{\infty}}\\
&  =sf^{\ast}(t)\\
&  \leq sf^{\ast}(s).
\end{align*}
Therefore,%
\begin{align*}
\left\Vert f\right\Vert _{(L^{1}(R^{n}),L^{\infty}(R^{n}))_{1/n,1}} &
\geq\int_{0}^{t}s^{-1/n}\{sf^{\ast\ast}(s)-sf^{\ast}(s)\}\frac{ds}{s}\\
&  =\int_{0}^{t}s^{1-1/n}\{f^{\ast\ast}(s)-f^{\ast}(s)\}\frac{ds}{s}.
\end{align*}
Consequently,
\[
\int_{0}^{t}s^{1-1/n}\{f^{\ast\ast}(s)-f^{\ast}(s)\}\frac{ds}{s}\leq c_{n}%
\int_{0}^{t}\left\vert \nabla f\right\vert ^{\ast}(s)ds,
\]
an inequality first shown in \cite{mmp}.

\subsection{Bilinear Interpolation\label{secc:bilinear}}

In this section we show an extension of (\ref{intro3}) to a class of bilinear
operators that have a product or convolution like structure. These operators
were first introduced by O'Neil (cf. \cite[Exercise 5, page 76]{bl}).

Let $\vec{A},\vec{B},\vec{C},$ be Banach pairs, and let $\Pi$ be a bilinear
bounded operator such that%
\[
\Pi:\left\{
\begin{array}
[c]{cc}%
A_{0}\times B_{0} & \rightarrow C_{0}\\
A_{0}\times B_{1} & \rightarrow C_{1}\\
A_{1}\times B_{0} & \rightarrow C_{1}%
\end{array}
\right.  .
\]
For example, the choice $A_{0}=B_{0}=C_{0}=L^{\infty},$ and $A_{1}=B_{1}%
=C_{1}=L^{1},$ corresponds to a regular product operator $\Pi_{1}(f,g)=fg,$
while the choice $A_{0}=B_{0}=C_{1}=L^{1},$ and $A_{1}=B_{1}=C_{0}=L^{\infty}$
corresponds to a convolution operator $\Pi_{2}(f,g)=f\ast g$ on $R^{n},$ say.
The main boundedness result is given by (cf. \cite[Exercise 5, page 76]{bl}):
\begin{equation}
\left\Vert \Pi(f,g)\right\Vert _{\vec{C}_{\theta,r}}\leq\left\Vert
f\right\Vert _{\vec{A}_{\theta_{1},q_{1}r}}\left\Vert g\right\Vert _{\vec
{B}_{\theta_{2},q_{2}r}}, \label{bilineal}%
\end{equation}
where $\theta,\theta_{i}\in(0,1),\theta=\theta_{1}+\theta_{2},q\,_{i}%
,r\in\lbrack1,\infty],i=1,2,$ and $\frac{1}{r}\leq\frac{1}{q_{1}r}+\frac
{1}{q_{2}r}.$

\begin{example}
To illustrate the ideas, make it easy to compare results, and to avoid the tax
of lengthy computations with indices, we shall only model an extension of
(\ref{intro3}) for $\Pi_{1}$. Let us thus take $\vec{A}=\vec{B}=\vec{C}%
=(A_{0},A_{1}),$ and let $\vec{X}=(X_{0},X_{1})=(A_{1},A_{0}).$ Further, in
order to be able to use the quadratic form argument outlined in the
Introduction, we choose $q_{1}=q_{2}=2,$ $r=p,$ $\theta_{1}=\theta_{2}%
=\frac{\theta}{2}.$ Then, from (\ref{bilineal}) we get%
\begin{equation}
\left\Vert \Pi_{1}(f,g)\right\Vert _{\vec{X}_{1-\theta,p}}\preceq\left\Vert
f\right\Vert _{\vec{X}_{1-\frac{\theta}{2},2p}}\left\Vert g\right\Vert
_{\vec{X}_{1-\frac{\theta}{2},2p}}. \label{laqueprecisa}%
\end{equation}
Since $(1-\frac{1}{2})(1-\theta)+\frac{1}{2}1=1-\frac{\theta}{2},$ we can use
the reiteration theorem to write%
\[
\vec{X}_{1-\frac{\theta}{2},2p}=(\vec{X}_{1-\theta,p},X_{1})_{\frac{1}{2}%
,2p},
\]
and find that%
\begin{equation}
\left\Vert \Pi_{1}(f,g)\right\Vert _{\vec{X}_{1-\theta,p}}\preceq\left\Vert
f\right\Vert _{(\vec{X}_{1-\theta,p},X_{1})_{\frac{1}{2},2p}}\left\Vert
g\right\Vert _{(\vec{X}_{1-\theta,p},X_{1})_{\frac{1}{2},2p}}.
\label{retornari}%
\end{equation}
By the equivalence theorem\footnote{In this example we are not interested on
the precise dependence of the constants of equivalence.} (cf. (\ref{equiva})
above),%
\[
(\vec{X}_{1-\theta,p},X_{1})_{\frac{1}{2},2p}=(\vec{X}_{1-\theta,p}%
,X_{1})_{\frac{1}{2},2p}^{(2)}.
\]
Therefore,
\begin{align*}
\left\Vert \Pi_{1}(f,g)\right\Vert _{\vec{X}_{1-\theta,p}}  &  \preceq
\left\Vert f\right\Vert _{(\vec{X}_{1-\theta,p},X_{1})_{\frac{1}{2},2p}^{(2)}%
}\left\Vert g\right\Vert _{(\vec{X}_{1-\theta,p},X_{1})_{\frac{1}{2},2p}%
^{(2)}}\\
&  \preceq\left\Vert f\right\Vert _{\vec{X}_{1-\theta,p}}^{1/2}\left\Vert
f\right\Vert _{\left(  \vec{X}_{1-\theta,p},X_{1}\right)  _{1,\infty}^{(1)}%
}^{1/2}\left\Vert g\right\Vert _{\vec{X}_{1-\theta,p}}^{1/2}\left\Vert
g\right\Vert _{\left(  \vec{X}_{1-\theta,p},X_{1}\right)  _{1,\infty}^{(1)}%
}^{1/2},
\end{align*}
where in the last step we have used (\ref{bmo5}) applied to the pair $(\vec
{X}_{1-\theta,p},X_{1}).$ Consequently, by the quadratic formula, we finally
obtain
\[
\left\Vert \Pi_{1}(f,g)\right\Vert _{\vec{X}_{1-\theta,p}}\preceq\left\Vert
f\right\Vert _{\vec{X}_{1-\theta,p}}\left\Vert g\right\Vert _{\left(  \vec
{X}_{1-\theta,p},X_{1}\right)  _{1,\infty}^{(1)}}+\left\Vert f\right\Vert
_{\left(  \vec{X}_{1-\theta,p},X_{1}\right)  _{1,\infty}^{(1)}}\left\Vert
g\right\Vert _{\vec{X}_{1-\theta,p}}.
\]

\end{example}

\begin{remark}
The method above uses one of the more powerful tools of the real method: the
re-scaling of inequalities. However, in this case there is substantial
difficulty for the implementation of the result since techniques for the
actual computation of the space $\left(  \vec{X}_{1-\theta,p},X_{1}\right)
_{1,\infty}^{(1)}$ are not well developed at present time\footnote{It is an
interesting open problem to modify Holmstedt's method to be able to keep
track, in a nearly optimal way, both coordinates in the Gagliardo diagram,
when doing reiteration. For more on the computation of Gagliardo coordinate
spaces see the forthcoming \cite{mil3}.}$.$ Therefore we avoid the use of
(\ref{retornari}) and instead prove directly that if\footnote{To simplify the
computations we model the $L^{p}$ case here. Another simplification is that in
this argument we don't need to be fuzzy about constants.} $\theta=\frac{1}%
{p},$ we have
\begin{equation}
\left\Vert f\right\Vert _{\vec{X}_{1-\frac{\theta}{2},2p}}\preceq\left\Vert
f\right\Vert _{\vec{X}_{1-\theta,p}}^{1/2}\left\Vert f\right\Vert _{\vec
{X}_{1,\infty}^{(1)}}^{1/2}.\label{berkovich}%
\end{equation}
This given, applying (\ref{berkovich}) to both terms on the right hand side of
(\ref{laqueprecisa}) and then using the quadratic form argument we find%
\begin{equation}
\left\Vert \Pi_{1}(f,g)\right\Vert _{\vec{X}_{1-\theta,p}}\preceq\left\Vert
f\right\Vert _{\vec{X}_{1-\theta,p}}\left\Vert g\right\Vert _{\vec
{X}_{1,\infty}^{(1)}}+\left\Vert g\right\Vert _{\vec{X}_{1-\theta,p}%
}\left\Vert f\right\Vert _{\vec{X}_{1,\infty}^{(1)}}.\label{elcaso}%
\end{equation}
In the particular case of product operators and $L^{p}$ spaces, $1<p<\infty,$
(\ref{elcaso}) reads%
\[
\left\Vert fg\right\Vert _{L^{p}}\preceq\left\Vert f\right\Vert _{L^{p}%
}\left\Vert g\right\Vert _{L(\infty,\infty)}+\left\Vert g\right\Vert _{L^{p}%
}\left\Vert f\right\Vert _{L(\infty,\infty)},
\]
which should be compared with (\ref{intro3}) recalling that%
\[
\left\Vert f\right\Vert _{L(\infty,\infty)}\leq c\left\Vert f\right\Vert
_{BMO}.
\]

\end{remark}

\begin{proof}
(of (\ref{berkovich})). To simplify the notation we let $K(t,f;\vec{X})=K(t).$
Then,%
\begin{align*}
\left\Vert f\right\Vert _{\vec{X}_{1-\frac{\theta}{2},2p}}  &  =\left\{
\int_{0}^{\infty}\left(  K(s)s^{-(1-\frac{\theta}{2})}\right)  ^{2p}\frac
{ds}{s}\right\}  ^{1/2p}\\
&  \leq\left\{  \int_{0}^{t}\left(  K(s)s^{-(1-\frac{\theta}{2})}\right)
^{2p}\frac{ds}{s}\right\}  ^{1/2p}+\left\{  \int_{t}^{\infty}\left(
K(s)s^{-(1-\frac{\theta}{2})}\right)  ^{2p}\frac{ds}{s}\right\}  ^{1/2p}\\
&  =(1)+(2).
\end{align*}
We proceed to estimate each of these two terms starting with $(2):$%
\begin{align*}
(2)  &  =\left\{  \int_{t}^{\infty}\left(  K(s)s^{-(1-\theta)}\right)
^{p}s^{(1-\theta)p}s^{-(1-\frac{\theta}{2})p}\left(  K(s)s^{-(1-\frac{\theta
}{2})}\right)  ^{p}\frac{ds}{s}\right\}  ^{1/2p}\\
&  \leq\left\{  \sup_{s\geq t}s^{(1-\theta)}s^{-(1-\frac{\theta}{2})}\left(
K(s)s^{-(1-\frac{\theta}{2})}\right)  \right\}  ^{1/2}\left\{  \int
_{t}^{\infty}\left(  K(s)s^{-(1-\theta)}\right)  ^{p}\frac{ds}{s}\right\}
^{1/2p}\\
&  \leq\left\{  \sup_{s\geq t}\frac{K(s)}{s}\right\}  ^{1/2}\left\Vert
f\right\Vert _{\vec{X}_{1-\theta,p}}^{1/2}\\
&  =\left\{  \frac{K(t)}{t}\right\}  ^{1/2}\left\Vert f\right\Vert _{\vec
{X}_{1-\theta,p}}^{1/2}\\
&  =t^{-1/2}t^{(1-\theta)/2}\left\{  K(t)t^{-(1-\theta)}\right\}
^{1/2}\left\Vert f\right\Vert _{\vec{X}_{1-\theta,p}}^{1/2}\\
&  \leq t^{-1/p2}\left\Vert f\right\Vert _{\vec{X}_{1-\theta,p}}.
\end{align*}
Moreover,%
\begin{align*}
(1)  &  =\left\{  \int_{0}^{t}\left(  \frac{K(s)}{s}s^{-(1-\frac{\theta}%
{2})+1}\right)  ^{2p}\frac{ds}{s}\right\}  ^{1/2p}\\
&  \leq\left\{  \int_{0}^{t}\left(  \left[  \frac{K(s)}{s}-\frac{K(t)}%
{t}\right]  s^{-(1-\frac{\theta}{2})+1}\right)  ^{2p}\frac{ds}{s}\right\}
^{1/2p}+\left\{  \int_{0}^{t}\left(  \frac{K(t)}{t}s^{-(1-\frac{\theta}{2}%
)+1}\right)  ^{2p}\frac{ds}{s}\right\}  ^{1/2p}\\
&  =(a)+(b).
\end{align*}

The term $(b)$ is readily estimated:%
\begin{align*}
(b)  &  =\frac{K(t)}{t}\left\{  \int_{0}^{t}s^{\frac{\theta}{2}2p}\frac{ds}%
{s}\right\}  ^{1/2p}\\
&  \sim\frac{K(t)}{t}t^{1/2p}\\
&  =K(t)t^{-(1-\theta)}t^{(1-\theta)}t^{1/2p-1}\\
&  \preceq\left\Vert f\right\Vert _{\vec{X}_{1-\theta,p}}t^{-1/2p}.
\end{align*}
Next we use the familiar estimate%
\begin{align*}
\left[  \frac{K(s)}{s}-\frac{K(t)}{t}\right]   &  =\int_{s}^{t}\left[
\frac{K(s)}{s}-K^{\prime}(s)\right]  \frac{ds}{s}\\
&  \leq\left\Vert f\right\Vert _{\vec{X}_{1,\infty}}\log\frac{t}{s},
\end{align*}
to see that%
\begin{align*}
(a)  &  \leq\left\Vert f\right\Vert _{\vec{X}_{1,\infty}}\left\{  \int_{0}%
^{t}s\left(  \log\frac{t}{s}\right)  ^{2p}\frac{ds}{s}\right\}  ^{1/2p}\\
&  \preceq\left\Vert f\right\Vert _{\vec{X}_{1,\infty}^{(1)}}t^{1/2p}.
\end{align*}
Collecting estimates, we have%
\[
\left\Vert f\right\Vert _{\vec{X}_{1-\frac{\theta}{2},2p}}\preceq
t^{-1/p2}\left\Vert f\right\Vert _{\vec{X}_{1-\theta,p}}+\left\Vert
f\right\Vert _{\vec{X}_{1,\infty}^{(1)}}t^{1/2p}.
\]
Balancing the two terms on the right hand side we find%
\[
\left\Vert f\right\Vert _{\vec{X}_{1-\frac{\theta}{2},2p}}\preceq\left\Vert
f\right\Vert _{\vec{X}_{1-\theta,p}}^{1/2}\left\Vert f\right\Vert _{\vec
{X}_{1,\infty}^{(1)}}^{1/2},
\]
as we wished to show.
\end{proof}

\begin{remark}
It would be of interest to implement a similar model analysis for $\Pi
_{2}(f,g)=f\ast g.$ We claim that it is easy, however, we must leave the task
to the interested reader.
\end{remark}

\begin{remark}
The results of this section are obviously connected with Leibniz rules in
function spaces. This brings to mind the celebrated commutator theorem of
Coifman-Rochberg-Weiss \cite{cofro} which states that if $T$ is a
Calder\'{o}n-Zygmund operator (cf. \cite{cosa}), and $b$ is a $BMO$ function,
then $[T,b]f=bTf-T(bf),$ defines a bounded linear operator\footnote{In fact,
the boundedness of $[T,b]$ for all CZ operators implies that $b\in BMO$ (cf.
\cite{janson})}%
\[
\lbrack T,b]:L^{p}(R^{n})\rightarrow L^{p}(R^{n}),1<p<\infty.
\]
The result has been extended in many different directions. In particular, an
abstract theory of interpolation of commutators has evolved from it (cf.
\cite{cwkamiro}, \cite{rr}, for recent surveys). In this theory a crucial role
is played by a class of operators $\Omega$ such that, for bounded operators
$T$ on a given interpolation scale, the commutator $[T,\Omega]$ is also
bounded. The operators $\Omega$ often satisfy some of the functional equations
associated with derivation operators. At present time we know very little
about how this connection comes about or how to exploit it in concrete
applications (cf. \cite{cwmiro}, \cite{konmil}). More in keeping with the
topic of this paper, and in view of many possible interesting applications, it
would be of also of interest to study oscillation inequalities for commutators
$[T,\Omega]$ in the context of interpolation theory. In this connection we
should mention that in \cite{mamiwa}, the authors formulate a generalized
version of the Coifman-Rochberg-Weiss commutator theorem, valid in the context
of real interpolation, where in a suitable fashion the function space $W,$
introduced in \cite{misa}, plays the role of $BMO:$
\[
W=\{f\in L_{loc}^{1}(0,\infty):\sup_{t}\left\vert \frac{1}{t}\int_{0}%
^{t}f(s)ds-f(t)\right\vert <\infty\}.
\]
Obviously, $f\in L(\infty,\infty)$ iff $f^{\ast}\in W.$
\end{remark}

\section{A brief personal note on Cora Sadosky}

I still remember well the day (circa March 1977) that I met Mischa Cotlar and
Cora Sadosky at Mischa's apartment in Caracas (cf. \cite{milcot}). I was
coming from Sydney, en route to take my new job as a Visiting Professor at
Universidad de Los Andes, in Merida. And while, of course, I had heard a lot
about them, I did not know them personally. Mischa and I had exchanged some
correspondence (no e-mail those days!). I had planned to come to spend an
afternoon with the Cotlars, taking advantage of a 24 hour stopover while en
route to Merida, a colonial city in the Andes region of Venezuela. As it turns
out, my flight for the next day had been cancelled, and Mischa and Yani
invited me to stay over at their place.

When I came to the Cotlar's apartment, Corita and Mischa were working in the
dining room. Mischa gave me a brief explanation of what they were doing
mathematically. In fact, he dismissed the whole enterprise\footnote{I would
learn quickly that Mischa's modesty was legendary.}. I would learn much later
that what they were doing then, would turn out to be very innovative and
influential research\footnote{From this period I can mention \cite{cotlar},
\cite{cotlar1}, \cite{cotlar2}.}.

Corita, who also knew about my impending arrival, greeted me with something
equivalent to *Oh, so you really do exist*!\footnote{Existence here to be
taken in a non mathematical sense. Of course at point in time I did NOT
*exist* mathematically!} Being a *porte\~{n}o* myself, I quickly found
Corita's style quite congenial and the conversation took off. We formed a
friendship that lasted for as long as she lived.

As I later learned, most people called her Cora, but the Cotlars, and a few
other old friends, that knew her from childhood, called her Corita\footnote{In
Spanish *Corita* means little Cora..}. Having being introduced to her at the
Cotlar's home, I proceeded to call her Corita too...and that was the way it
would always be\footnote{Many many years later she told me that by then
everyone called her Cora, except Mischa and myself and that she would prefer
for me to call her Cora. I said, of course, Corita!..}.

By early 1979 I moved for one semester to Maracaibo and afterwards to
Brasilia. Mischa was very helpful connecting me first with Jorge Lebowitz
(Maracaibo), and then with Djairo Figuereido (Brasilia). In the mean time
Corita herself had moved to the US, where we met again, at an AMS Special
Session in Harmonic Analysis.

At the time I had a visiting position at UI at Chicago, and was trying to find
a tenure track job. She took an interest in my situation, and gave me very
useful suggestions on the job hunting process. She would remain very helpful
throughout my career. In particular, when she learned\footnote{She had a
membership at IAS herself that year}, on her own, about my application for a
membership at the Institute for Advanced Study, in 1984, she supported my
case.... I did not know this until she called me to let me know that my
application had been accepted!

Corita and I met again many times over the course of the years. There were
conferences on Interpolation Theory\footnote{Harmonic analysts trained in the
sixties had a special place in their hearts for Interpolation theory. It was
after all a theory to which the great masters of the Chicago school (e.g.
Calder\'{o}n, Stein, Zygmund) had made fundamental contributions.} in Lund and
Haifa, on Harmonic Analysis in Washington D.C., Madrid and Boca Raton, there
were Chicago meets to celebrate various birthdays of Alberto Calder\'{o}n, a
special session on Harmonic Analysis in Montreal, etc, etc. Vanda and I went
to have dinner with Corita and her family, when we all coincided during a
visit to Buenos Aires in l985. She was instrumental in my participation at
Mischa's 80 birthday Conference in Caracas, 1994. We wrote papers for books
that each of us edited. At some point in time she and her family came to visit
us in Florida.

A few months before she passed, I was helping her, via e-mail, with the
texting\footnote{I mean using TeX!} of a paper of hers about Mischa Cotlar.

Probably the best way I have to describe Corita is to say that she was a force
of nature. She was a brilliant mathematician, with an intense but charming
personality. Having had to endure herself exile, discrimination, and very
difficult working conditions, she was very sensitive to the plight of others.
I am sure many of the testimonies in this book will describe how much she
helped to provide opportunities for younger mathematicians to develop their careers.

Some things are hard to change. As much as I could not train myself to call
her *Cora*, in our relationship she was always \textquotedblleft big sister".
I will always be grateful to her, for all her help and her friendship.

I know that the space $BMO$ had a very special place in her mathematical
interests and, indeed, $BMO$ spaces appear in many considerations throughout
her works. For this very reason, and whatever the merits of my small
contribution, I have chosen to dedicate this note on $BMO$ inequalities to her memory.

\textbf{Acknowledgement}. I am grateful to Sergey Astashkin and Michael Cwikel
for useful comments that helped improve the quality of the paper. Needless to
say that I remain fully responsible for the remaining shortcomings.


\begin{thebibliography}{99}                                                                                               %


\bibitem {aalto}D. Aalto, \textsl{Weak }$L^{\infty}$\textsl{ and }%
$BMO$\textsl{ in metric spaces}. Boll. Unione Mat. Ital. \textbf{5} (2012), 369--385.

\bibitem {alvamil}J. Alvarez and M. Milman, \textsl{Spaces of Carleson
measures: duality and interpolation}. Ark. Mat. \textbf{25} (1987), 155--174.

\bibitem {arz}N. Arcozzi, R. Rochberg, E. Sawyer, and B. D. Wick, Brett,
\textsl{Potential theory on trees, graphs and Ahlfors-regular metric spaces}.
Potential Anal. \textbf{41} (2014), 317--366.

\bibitem {ash}S. Astashkin and K. Lykov, \textsl{Extrapolation description of
rearrangement invariant spaces and related problems}. Banach and function
spaces III, pp. 1--52, Yokohama Publ., Yokohama, 2011.

\bibitem {bagby}R. J. Bagby and D. S. Kurtz, \textsl{A rearranged
good-}$\lambda$\textsl{ inequality}. Trans. Amer. Math. Soc. \textbf{293}
(1986), 71-81.

\bibitem {bakr}D. Bakry, T. Coulhon, M. Ledoux and L. Saloff-Coste,
\textsl{Sobolev inequalities in disguise}. Indiana Univ. Math. J. \textbf{44}
(1995) 1033--1074.

\bibitem {bmr}J. Bastero, M. Milman and F. J. Ruiz Blasco, \textsl{A note on
}$L(\infty,q)$\textsl{ spaces and Sobolev embeddings}. Indiana Univ. Math. J.
\textbf{52} (2003), 1215--1230.

\bibitem {beka}C. Bennett, \textsl{Banach function spaces and interpolation
methods I. The abstract theory}. J. Funct. Anal. \textbf{17} (1974), 409--440.

\bibitem {bedesa}C. Bennett, R. A. DeVore and R. Sharpley, \textsl{Weak-L}%
$^{\infty}$\textsl{ and BMO}. Ann. of Math \textbf{113} (1981), 601--611.

\bibitem {bs}C. Bennett and R. Sharpley, \textsl{Interpolation of operators}.
Pure and Applied Mathematics \textbf{129}, Academic Press, Inc., Boston, MA, 1988.

\bibitem {besa}C. Bennett and R. Sharpley, \textsl{Weak-type inequalities for
}$H^{p}$\textsl{ and BMO}. Harmonic analysis in Euclidean spaces (Proc.
Sympos. Pure Math., Williams Coll., Williamstown, Mass., 1978), Part 1, pp.
201--229, Proc. Sympos. Pure Math., XXXV, Amer. Math. Soc., Providence, R.I., 1979.

\bibitem {bl}J. Bergh and J. L\"{o}fstrom, \textsl{Interpolation Spaces}.
Springer-Verlag, New York, Berlin, 1976.

\bibitem {brez}H. Brezis and S. Wainger, \textsl{A note on limiting cases of
Sobolev embeddings and convolution inequalities}. Comm. Partial Diff. Eq.
\textbf{5} (1980), 773-789.

\bibitem {ca}A. P. Calder\'{o}n, \textsl{Intermediate spaces and
interpolation, the complex method}. Studia Math. \textbf{24} (1964), 113-190.

\bibitem {ca1}A.P. Calder\'{o}n, \textsl{Spaces between }$L^{1}$\textsl{ and
}$L^{\infty}$\textsl{ and the theorem of Marcinkiewicz}. Studia Math.
\textbf{26} (1966), 273--299.

\bibitem {bu}D. L. Burkholder, \textsl{Distribution function inequalities for
martingales}. Ann. of Prob. \textbf{1} (1973), 19-42.

\bibitem {bg}D. L. Burkholder and R. F. Gundy, \textsl{Extrapolation and
interpolation of quasi-linear operators on martingales}. Acta Math.
\textbf{124} (1970), 249-304.

\bibitem {bg1}D. L. Burkholder and R. F. Gundy, \textsl{Distribution function
inequalities for the area integrals}. Studia Math. \textbf{44 }(1972), 527-544.

\bibitem {chenzu}J. Chen, and X. Zhu, \textsl{A note on BMO and its
application}. J. Math. Anal. Appl. \textbf{303} (2005), 696--698.

\bibitem {cobos}F. Cobos and M. Milman, \textsl{On a limit class of
approximation spaces}. Num. Funct. Anal. Opt. \textbf{11} (1990), 11-31.

\bibitem {cof}R. R. Coifman and C. Fefferman, \textsl{Weighted norm
inequalities for maximal functions and singular integrals}. Studia Math.
\textbf{51} (1974), 241-250.

\bibitem {cofro}R. R. Coifman, R. Rochberg and G. Weiss, \textsl{Factorization
theorems for Hardy spaces in several variables}. Ann. Math. 103 (1976), 611--635.

\bibitem {cotlar}M. Cotlar and C. Sadosky, \textsl{Transform\'{e}e de Hilbert,
Th\'{e}or\`{e}me de Bochner et le Probl\`{e}me des Moments I}. C.R. Acad. Sci.
Paris Ser. A-B \textbf{285} (1977).

\bibitem {cotlar1}M. Cotlar and C. Sadosky, \textsl{Transform\'{e}e de
Hilbert, Th\'{e}or\`{e}me de Bochner et le Probl\`{e}me des Moments II}. C.R.
Acad. Sci. Paris Ser. A-B \textbf{285} (1977).

\bibitem {cotlar2}M. Cotlar and C. Sadosky, \textsl{On the Helson-Szeg\"{o}
Theorem and a Related Class of Modified Toeplitz Kernels}. In
\textquotedblleft Harmonic Analysis in Euclidean Spaces", Proc. Symp Pure
Math. Amer. Math. Soc. \textbf{35} (1979), 383-407.

\bibitem {cwja}M. Cwikel and S. Janson, \textsl{Real and complex interpolation
methods for finite and infinite families of Banach spaces}. Adv. Math.
\textbf{66} (1987) 234--290.

\bibitem {cwbjmm0}M. Cwikel, B. Jawerth and M. Milman, \textsl{On the
Fundamental Lemma of Interpolation Theory}. J. Approx. Theory \textbf{60}
(1990), 70-82.

\bibitem {cwbjmm}M. Cwikel, B. Jawerth and M. Milman, \textsl{A note on
extrapolation of inequalities}. Preprint, 2010.

\bibitem {cwkamiro}M. Cwikel, N. Kalton, M. Milman, and R. Rochberg, \textsl{A
unified theory of commutator estimates for a class of interpolation methods}.
Adv. Math. \textbf{169} (2002), 241-312.

\bibitem {cwmiro}M. Cwikel, M. Milman, and R. Rochberg, \textsl{An
introduction to Nigel Kalton's work on differentials of complex interpolation
processes for K\"{o}the spaces}. Arxiv: http://arxiv.org/pdf/1404.2893.pdf .

\bibitem {cwsa}M. Cwikel and Y. Sagher, \textsl{Weak type classes}. J. Funct.
Anal. \textbf{52 }(1983), 11-18.

\bibitem {css}M. Cwikel, Y. Sagher and P. Shvartsman, \textsl{A
geometrical/combinatorical question with implications for the JohnNirenberg
inequality for BMO functions}. Banach Center Publ. \textbf{95} (2011), 45-53.

\bibitem {fest}C. Fefferman and E. M. Stein, $H^{p}$\textsl{ spaces of several
variables}. Acta Math. \textbf{129} (1972), 137--193.

\bibitem {garsiagrenoble}A. M. Garsia and E. Rodemich, \textsl{Monotonicity of
certain functionals under rearrangements}. Ann. Inst. Fourier (Grenoble)
\textbf{24} (1974), 67-116.

\bibitem {haj}P. Hajlasz,\textsl{\ Sobolev inequalities, truncation method,
and John domains}. Papers in Analysis, Rep. Univ. Jyv\"{a}skyl\"{a} Dep. Math.
Stat. 83, Univ. Jyv\"{a}skyl\"{a}, Jyv\"{a}skyl\"{a}, 2001, pp 109--126.

\bibitem {ha}K. Hansson, \textsl{Imbedding theorems of Sobolev type in
potential theory}. Math Scand \textsl{45} (1979), 77-102.

\bibitem {herz}C. Herz, \textsl{A best possible Marcinkiewicz theorem with
applications to martingales}. Mc Gill University, 1974.

\bibitem {he}C. Herz, \textsl{An interpolation principle for martingale
inequalities}. J. Funct. Anal. \textbf{22} (1976), 1-7.

\bibitem {holm0}T. Holmstedt, \textsl{Interpolation of quasi-normed spaces}.
Math. Scand. \textbf{26} (1970), 177-199.

\bibitem {holm}T. Holmstedt, \textsl{Equivalence of two methods of
interpolation}. Math. Scand. \textbf{18} 1966, 45--52.

\bibitem {hon}P. Honz\'{\i}k and B. J. Jaye, \textsl{On the good-}$\lambda
$\textsl{ inequalities for nonlinear potentials}. Proc. Amer. Math. Soc.
\textbf{140} (2012), 4167--4180.

\bibitem {janson}S. Janson, \textsl{Mean oscillation and commutators of
singular integral operators}. Ark. Math. \textbf{16} (1978), 263-270.

\bibitem {ja}B. Jawerth, \textsl{The K-functional for }$H^{1}$\textsl{ and
}$BMO.$ Proc. Amer. Math. Soc. \textbf{92} (1984), 67--71.

\bibitem {jm1}B. Jawerth and M. Milman, \textsl{Interpolation of weak type
spaces}. Math. Z. \textbf{201} (1989), 509--519.

\bibitem {jm}B. Jawerth and M. Milman, \textsl{Extrapolation theory with
applications}. Mem. Amer. Math. Soc. \textbf{89 }(1991), no. 440.

\bibitem {jami}B. Jawerth and M. Milman, \textsl{Lectures on
optimization,image processing, and interpolation Theory.} In "Function Spaces,
Inequalities and Interpolation", Lectures on Optimization, Image Processing
and Interpolation Theory, Spring School on Analysis held at Paseky (2007) (in preparation).

\bibitem {jamig}B. Jawerth and M. Milman, \textsl{Gagliardo discretizations
and best approximation} (with an Appendix by Emil Cornea). To appear.

\bibitem {johnN}F. John and L. Nirenberg, \textsl{On functions of bounded mean
oscillation}. Comm. Pure Appl. Math. \textbf{14} (1961), 415--426.

\bibitem {kalis}J. Kalis and M. Milman, \textsl{Symmetrization and sharp
Sobolev inequalities in metric spaces}. Rev. Mat. Complutense \textbf{22}
(2009), 499-515.

\bibitem {kami}G. E. Karadzhov and M. Milman, \textsl{Extrapolation theory:
new results and applications}. J. Approx. Theory \textbf{133} (2005), 38--99.

\bibitem {kamixi}G. E. Karadzhov, M. Milman and J. Xiao, \textsl{Limits of
higher-order Besov spaces and sharp reiteration theorems}. J. Funct. Anal.
\textbf{221} (2005), 323--339

\bibitem {kol}V. I. Kolyada, \textsl{Rearrangements of functions and embedding
theorems}. Russ. Math. Surv. \textbf{44} (1989), 73--117.

\bibitem {konmil}H. K\"{o}nig and V. Milman, \textsl{Characterizing the
derivative and the entropy function by the Leibniz rule}, \textsl{with an
appendix by D. Faifman}. J. Funct. Anal. \textbf{261} (2011), 1325--1344.

\bibitem {kowa}H. Kozono and H. Wadade, \textsl{Remarks on Gagliardo-Nirenberg
type inequality with critical Sobolev space and BMO}. Math. Z. \textbf{259}
(2008), 935--950.

\bibitem {kowa1}H. Kozono and Y. Taniuchi, \textsl{Bilinear estimates in BMO
and the Navier-Stokes equations}. Math. Z. \textbf{235} (2000), 173--194.

\bibitem {krular}N. Krugljak, L. Maligranda and L. E. Persson, \textsl{An
elementary approach to the fractional Hardy inequality}. Proc. Amer. Math.
Soc. \textbf{128} (2000), 727-734.

\bibitem {kru}N. Kruglyak, Y. Sagher and P. Shvartsman, \textsl{Weak-type
operators and the strong fundamental lemma of real interpolation theory}.
Studia Math. \textbf{170} (2005), 173-201.

\bibitem {kurtz}D. S. Kurtz, \textsl{Operator estimates using the sharp
function}, Pacific J. Math. \textbf{139} (1989), 267-277.

\bibitem {le}M. Ledoux, \textsl{Isop\'{e}rim\'{e}trie et in\'{e}galit\'{e}es
de Sobolev logarithmiques gaussiennes}. C. R. Acad. Sci. Paris \textbf{306}
(1988), 79-92.

\bibitem {leo}G. Leoni, \textsl{A First Course in Sobolev Spaces}. Graduate
Studies in Mathematics, vol. \textbf{105}. American Mathematical Society, New
York, 2009.

\bibitem {mamipams}J. Mart\'{\i}n M. Milman, \textsl{Symmetrization
inequalities and Sobolev embeddings}. Proc. Amer. Math. Soc. \textbf{134
}(2006), 2335--2347.

\bibitem {mami}J. Martin and M. Milman, \textsl{A note on Sobolev inequalities
and limits of Lorentz spaces.} L. De Carli et al. (eds), Interpolation theory
and applications, pp. 237--245, \ Contemp. Math. \textbf{445}, Amer. Math.
Soc., 2007.

\bibitem {mamicoc}J. Martin and M. Milman, \textsl{Sobolev inequalities,
rearrangements, isoperimetry and interpolation spaces}. C. Houdre et al.
(eds), Concentration, functional inequalities and isoperimetry, pp. 167-193,
Contemp. Math. \textbf{545}, Amer. Math. Soc., 2011.

\bibitem {mamijfa}J. Martin and M. Milman, \textsl{Isoperimetry and
Symmetrization for Logarithmic Sobolev inequalities}. J. Funct. Anal.
\textbf{256} (2009), 149-178.

\bibitem {mamica}J. Martin and M. Milman, \textsl{Towards a Unified Theory of
Sobolev Inequalities}. A. M. Stokolos et al. (eds.), Special Functions,
Partial Differential Equations, and Harmonic Analysis, Springer Proceedings in
Mathematics \& Statistics \textbf{108}, 163-201, 2014.

\bibitem {mmp}J. Martin, M. Milman and E. Pustylnik, \textsl{Sobolev
Inequalities: Symmetrization and Self Improvement via truncation}. J. Funct.
Anal. \textbf{252} (2007), 677-695.

\bibitem {mamiwa}J. Martin and M. Milman, \textsl{An abstract
Coifman-Rochberg-Weiss commutator theorem}. In L. De Carli, K. Kazarian and M.
Milman (eds), Classical Analysis and Applications: A conference in honor of
Daniel Waterman, World Scientific, 141-160, 2008.

\bibitem {mastylo}M. Mastylo and M. Milman, \textsl{Interpolation of real
method spaces via some ideals of operators}. Studia Math. \textbf{136} (1999), 17-35.

\bibitem {maz}V. G. Maz'ya, \textsl{Sobolev Spaces with applications to
elliptic partial differential equations}. Second, revised and augmented
edition. Grundlehren der Mathematischen Wissenschaften [Fundamental Principles
of Mathematical Sciences], \textbf{342}. Springer, Heidelberg, 2011.

\bibitem {mac}D. McCormick, J. Robinson and J. L. Rodrigo, \textsl{Generalised
Gagliardo-Nirenberg inequalities using weak Lebesgue spaces and BMO}. Milan J.
Math. \textbf{81} (2013), 265--289.

\bibitem {restaurant}R. Metzler and J. Klafter, \textsl{The restaurant at the
end of the random walk: recent developments in the description of anomalous
transport by fractional dynamics}. J. of Physics A \textbf{37} (2004), R161--R208.

\bibitem {milcot}M. Milman, \textsl{Remembering Mischa}. In \textquotedblleft
An afternoon in honor of Mischa Cotlar", Univ. New Mexico, 2007, http://www.math.unm.edu/conferences/10thAnalysis/testimonies.html.

\bibitem {milta}M. Milman, \textsl{Notes on limits of Sobolev spaces and the
continuity of interpolation scales}. Trans. Amer. Math. Soc. \textbf{357
}(2005), 3425--3442.

\bibitem {mil}M. Milman, \textsl{A note on reversed Hardy inequalities and
Gehring's lemma}. Comm. Pure Appl. Math. \textbf{50} (1997), 311--315.

\bibitem {mil1}M. Milman, \textsl{Self improving inequalities via penalty
methods}. Paper presented at High Dimensional Probability VII, Corsica, 2014,
manuscript in preparation.

\bibitem {mil2}M. Milman, \textsl{Extrapolation: three lectures in Barcelona}.
Lecture Notes based on a 2013 course given at Univ. Autonoma Barcelona, in preparation.

\bibitem {mil3}M. Milman, \textsl{BMO: Marcinkiewicz spaces, the E functional
and Reiteration}, in preparation.

\bibitem {milpu}M. Milman and E. Pustylnik, \textsl{On sharp higher order
Sobolev embeddings}. Commun. Contemp. Math. \textbf{6} (2004), 495-511.

\bibitem {misa}M. Milman and Y. Sagher, \textsl{An interpolation theorem}.
Ark. Mat. \textbf{22} (1984), 33--38.

\bibitem {nev}J. Neveu, \textsl{Martingales a temps discret}. Dunod, Paris, 1972.

\bibitem {ov}V.I. Ovchinnikov, \textsl{The method of orbits in interpolation
theory}. Mathematical Reports \textbf{1} (1984), 349--516.

\bibitem {pus}E. Pustylnik, \textsl{On a rearrangement-invariant function set
that appears in optimal Sobolev embeddings}. J. Math. Anal. Appl. \textbf{344}
(2008) 788--798.

\bibitem {ren}P. F. Renaud, \textsl{A reversed Hardy inequality}. Bull.
Austral. Math. Soc. \textbf{34}, 1986, 225--232.

\bibitem {rr}R. Rochberg, \textsl{Uses of commutator theorems in analysis}.
Interpolation theory and applications, 277--295, Contemp. Math. \textbf{445},
Amer. Math. Soc., Providence, RI, 2007.

\bibitem {rota}G. C. Rota, \textsl{Ten lessons I wish I had been taught}. http://www.ams.org/notices/199701/comm-rota.pdf

\bibitem {cosa}C. Sadosky, \textsl{Interpolation of operators and singular
integrals. An introduction to harmonic analysis}. Monographs and Textbooks in
Pure and Applied Math. \textbf{53}, Marcel Dekker, Inc., New York, 1979.

\bibitem {sag}Y. Sagher, \textsl{An application of the approximation
functional in interpolation theory}. Conference on harmonic analysis in honor
of Antoni Zygmund, Vol. I, II (Chicago, Ill., 1981), 802--809, Wadsworth Math.
Ser., Wadsworth, Belmont, CA, 1983.

\bibitem {sasch}Y. Sagher and P. Shvartsman, \textsl{Rearrangement-Function
Inequalities and InterpolationTheory. }J. of Appr. Th. \textbf{119} (2002), 214--251.

\bibitem {talenti}G. Talenti, \textsl{Elliptic equations and rearrangements}.
Ann. Scuola Norm. Sup. Pisa Cl. Sci. \textbf{4} (1976), 697--718.

\bibitem {tru}N. S. Trudinger, \textsl{On imbeddings into Orlicz spaces and
some application}s. J. Math. Mech. \textbf{17} (1967), 473--483.

\bibitem {xiao}J. Xiao, \textsl{A Reverse BMO-Hardy Inequality}. Real Anal.
Exchange \textbf{25} (1999), 673-678.

\bibitem {xiao1}J. Xiao, \textsl{Gaussian BV capacity}. Adv. Cal. Var. Mar.
2015, on line.
\end{thebibliography}
\end{document}